\documentclass[12pt,reqno]{amsart}
\usepackage[latin1]{inputenc}
\usepackage{amssymb}
\usepackage{bbold}
\usepackage[T1]{fontenc}
\usepackage{multirow}

\usepackage{enumerate,color,graphicx}
\usepackage[all]{xy}
\usepackage{comment}
\usepackage{hyperref}

\newcounter{saveenumerate}
\makeatletter
\newcommand{\enumeratext}[1]{%
\setcounter{saveenumerate}{\value{enum\romannumeral\the\@enumdepth}}
\end{enumerate}
#1
\begin{enumerate}
\setcounter{enum\romannumeral\the\@enumdepth}{\value{saveenumerate}}%
}
\makeatother
\definecolor{darkgreen}{RGB}{65,165,118}
\definecolor{darkred}{RGB}{180,0,0}
\definecolor{darkblue}{RGB}{0,0,180}
\definecolor{ocre}{RGB}{204,119,34}     
\definecolor{marro}{RGB}{103,64,58}

\long\def\candrop#1{{\noindent\color{darkred}\textit{{\small{[#1]}}}}}

\newcommand{\mynote}[1]{%
  \ifshowmynotecmd
   \noindent\begingroup\footnotesize\textcolor{blue}{[Carles: #1]}\endgroup%
  \fi}
\newif\ifshowmynotecmd \showmynotecmdtrue

\newenvironment{blue}{\color{darkblue}\ttfamily\footnotesize}{}

\DeclareMathAlphabet\EuR{U}{eur}{m}{n}
\SetMathAlphabet\EuR{bold}{U}{eur}{b}{n}

\newcommand{\curs}{\EuR}

\newcommand{\Sets}{\curs{Sets}}

\newcommand{\domd}{\mathbf{D}}

\newcommand{\loce}{\text{\bbfamily E}}
\newcommand{\loch}{\text{\bbfamily H}}
\newcommand{\lock}{\text{\bbfamily K}}
\newcommand{\locl}{\text{\bbfamily L}}
\newcommand{\locm}{\text{\bbfamily M}}
\newcommand{\locn}{\text{\bbfamily N}}

\newcommand{\defeq}{\overset{\text{\textup{def}}}{=}}
\newcommand{\defin}{\defeq}

\newcommand{\Id}{\textup{Id}}

\newcommand{\cala}{\mathcal{A}}
\newcommand{\calb}{\mathcal{B}}
\newcommand{\calc}{\mathcal{C}}

\newcommand{\cale}{\mathcal{E}}

\newcommand{\calm}{\mathcal{M}}

\newcommand{\Map}{\operatorname{Map}\nolimits}

\newcommand{\Hom}{\operatorname{Hom}\nolimits}

\newcommand{\Aut}{\operatorname{Aut}\nolimits}
\newcommand{\aut}{\operatorname{aut}\nolimits}
\newcommand{\Out}{\operatorname{Out}\nolimits}

\newcommand{\Mor}{\operatorname{Mor}\nolimits}

\newcommand{\nv}[1]{|#1|} 

\newcommand{\cj}[1]{\7{c}_{#1}}    

\newcommand{\frakx}{\mathfrak{X}}

\newcommand{\incl}{\textup{incl}}
\newcommand{\proj}{\operatorname{proj}\nolimits}

\newcommand{\op}{^{\textup{op}}}

\let\oldcirc=\circ
\renewcommand{\circ}{\mathchoice
    {\mathbin{\scriptstyle\oldcirc}}{\mathbin{\scriptstyle\oldcirc}}
    {\mathbin{\scriptscriptstyle\oldcirc}}
    {\mathbin{\scriptscriptstyle\oldcirc}}}

\newcommand{\conj}[2]{\{\,#1\,|\,#2\,\}}

\newcommand{\longleft}[1]{\;{\leftarrow%
\count255=0 \loop \mathrel{\mkern-6mu}%
    \relbar\advance\count255 by1\ifnum\count255<#1\repeat}\;}
\newcommand{\longright}[1]{\;{\count255=0 \loop \relbar\mathrel{\mkern-6mu}%
    \advance\count255 by1\ifnum\count255<#1\repeat\rightarrow}\;}
\newcommand{\Right}[2]{\overset{#2}{\longright#1}}
\newcommand{\RIGHT}[3]{\mathrel{\mathop{\kern0pt\longright#1}
        \limits^{#2}_{#3}}}
\newcommand{\Left}[2]{{\buildrel #2 \over {\longleft#1}}}
\newcommand{\LEFT}[3]{\mathrel{\mathop{\kern0pt\longleft#1}\limits^{#2}_{#3}}
}
\newcommand{\dRIGHT}[3]{\mathrel{%
   \mathop{\vcenter{\baselineskip=0pt\hbox{$\kern0pt\longright#1$}%
   \hbox{$\kern0pt\longright#1$}}}\limits^{#2}_{#3}}}
\newcommand{\LRIGHT}[3]{\mathrel{%
   \mathop{\vcenter{\baselineskip=0pt\hbox{$\kern0pt\longleft#1$}%
   \hbox{$\kern0pt\longright#1$}}}\limits^{#2}_{#3}}}
\newcommand{\RLEFT}[3]{\mathrel{%
   \mathop{\vcenter{\baselineskip=0pt\hbox{$\kern0pt\longright#1$}%
   \hbox{$\kern0pt\longleft#1$}}}\limits^{#2}_{#3}}}
\newcommand{\onto}[1]{\;{\count255=0 \loop \relbar\joinrel
    \advance\count255 by1
    \ifnum\count255<#1 \repeat \twoheadrightarrow}\;}

\newcommand{\con}[2]{\{\, {#1} \,|\,  {#2}  \,\}}
\newcommand{\Con}[2]{\bigl\{\,  {#1}  \,\bigm|\,  {#2}   \,\bigr\}}

\newcommand{\calaut}{\operatorname{\cala ut}\nolimits}

\newcommand{\Der}{\operatorname{Der}\nolimits}

\newcommand{\4}[1]{\widehat{#1}}
\newcommand{\5}[1]{\widetilde{#1}}
\newcommand{\7}[1]{\boldsymbol{#1}}
\newcommand{\9}[1]{{}^{#1}\kern-1pt{}}

 \newcounter{enumi_saved}

     \newcounter{enumii_saved}

\theoremstyle{plain}
\newtheorem{thm}{Theorem}[section]
\newtheorem{prop}[thm]{Proposition}

\newtheorem{lmm}[thm]{Lemma}
\newtheorem{cor}[thm]{Corollary}

\theoremstyle{definition}
\newtheorem{rmk}[thm]{Remark}{\rm}
\newtheorem{defi}[thm]{Definition}
\newtheorem{expl}[thm]{Example}{\rm}
{\rm}
{\rm}
\newtheorem{nota}[thm]{Notation}{\rm}

\numberwithin{equation}{section}

\title{An extension theory for partial groups}
\author{Carles Broto and Alex Gonzalez}

\begin{document}

\begin{abstract}
Partial groups are a natural generalization of discrete groups recently introduced by Chermak in connection with the theory of fusion systems. In this paper we develop an extension theory for partial groups based in the classical 
theory of fibre bundles of simplicial sets. 
\end{abstract}

\maketitle                                      

The notion of partial group was introduced by Chermak \cite{Chermak} as an alternative to linking 
and  transporter systems \cite{BLO2,OV1} that might be more appealing to group theorists. 
Roughly speaking, a partial group is a set with an associative product which is only defined on 
certain tuples of elements and an inversion. The aim of this paper is to establish an
extension theory for partial groups that will be useful in order to understand more sophisticated 
structures \cite{BG2,Go}.

The extension theory of partial groups will be based on the theory 
of fibre bundles of simplicial sets \cite{BGM,Curtis,GJ,May}.
Indeed, to a partial group one may associate a simplicial set, 
built as a bar construction (cf.\ Example~\ref{BG})
which in turn determines explicitly the original partial group. 
Simplicial sets obtained in this way are 
easily characterized. For this reason we have found it appropriate for this paper
to define a partial group as a simplicial set 
with a unique vertex and higher simplices uniquely determined by their spines 
(see Definition~\ref{def.part.mon} for details). Likewise, homomorphisms between partial groups are arbitrary 
simplicial maps. 
It becomes clear that the bar construction provides an 
equivalence between this definition and Chermak's definition \cite[2.1]{Chermak}. 

Given a partial group $\locm$ we identify two characteristic subgroups, $N(\locm)$ and $Z(\locm)$. The first 
consists of all elements that define inner automorphisms of $\locm$, while $Z(\locm)$ is the subgroup of those 
for which the given inner automorphism is the identity map. $Z(\locm)$ is the center 
of $\locm$ and we call $N(\locm)$ the normalizer of $\locm$. 
We identify the outer automorphisms
$\Out(\locm)$  with the automorphisms of $\locm$ modulo 
$N(\locm)$, and we get an exact sequence of finite groups
\begin{equation*}
 1\Right2{} Z(\locm)  \Right2{} N(\locm)  \Right2{} \Aut(\locm)  \Right2{} \Out(\locm)  \Right2{} 1
\end{equation*}
(see Proposition~\ref{Nloclgroup}). 
This allows the description of the simplicial monoid of self-equivalences of a partial group $\locm$, that 
we denote $\aut(\locm)$. A self-equivalence is really an automorphism of $\locm$ and we obtain 
\begin{equation}\label{aut}
 \pi_i(\aut(\locm)) \cong \left\{
\begin{array}{ll}
\Out(\locm), & i = 0\\
Z(\locm), & i = 1\\
\{0\}, & i \geq 2\\
\end{array}
\right.
\end{equation}

In this more general case, as it happens in the case of finite groups, the simplicial set of self-equivalences has at most two non-trivial
homotopy groups. This fact is the ultimate reason why the classification of extensions of partial groups resembles the classical results 
for extensions of finite groups. 

We define an extension of partial groups to be a fibre bundle $\locm\Right0{}\loce\Right0{}\loch$. 
We first obtain a description 
of $\loce$ in terms of $\locm$ and $\loch$  and then show that $\loce$ is indeed a partial 
group if $\locm$ and $\loch$ were so (see Theorem~\ref{twistedn-simpl}). 
Fixing $\locm$,  $\loch$, and an outer action of $\loch$ on $\locm$, that is, a homomorphism 
$\alpha\colon\loch\Right1{}\Out(\locm)$, 
 we show that the set $\cale(\locm,\loch,\alpha)$ of equivalence classes of extensions is nonempty
provided a certain obstruction class $[\kappa] \in H^3(\loch; Z(\locm))$ vanishes.
Furthermore, assuming  $\cale(\locm,\loch,\alpha)\neq\emptyset$, the group 
$H^2(\loch; Z(\locm))$ acts freely and transitively on $\cale(\locm,\loch,\alpha)$ 
(see Theorem~\ref{classext}).

We also establish a classification theorem for sections of extensions. We  distinguish 
between arbitrary sections and regular sections. Given an extension
$\locm \Right1{} \loce \Right1{\tau} \loch$,  a section is called regular if it has image in $N(\loce)$. 
An extension that admits a regular section is called regular split. Regular split extensions are 
isomorphic to semi-direct products. 
We show that equivalence classes of sections of regular split extensions are classified by the first non-abelian cohomology group
$H^1(\loch; \locm)$. 
\medskip

\textbf{Organization of the paper.} In Section \ref{Spg} we review some basic concepts related to simplicial sets. 
Most of the work in this section is devoted to analyze homotopies between homomorphisms of partial groups, 
as they play a key role in this paper. In Section 3 we analyze the space of self-homotopy equivalences of a partial
group and we prove equation \eqref{aut}.
Section \ref{Sfibund} develops the theme of fibre bundles and extensions of partial groups. 
Finally, in Section \ref{regsplit} we introduce regular split extensions of partial groups and 
we prove the classification theorem for sections.

\medskip

\textbf{Acknowledgements.} The authors would like to express their gratitude to Andy Chermak for his help and insight. Their recognition is also extended to the Departments of Mathematics of both Kansas State University and the Universitat Aut\`onoma de Barcelona, as well as the Barcelona Algebraic Topology Group for their hospitality and their support to the authors throughout the preparation of this work. Both authors were partially supported by the 
MINECO-FEDER grants MTM2013-42293-P and  MTM2016-80439.


\section{Simplicial sets}
This section contains a brief discussion about simplicial sets focused on our needs. We refer the interested reader to \cite{May,GJ} for the general theory of simplicial sets. Let $\Delta$ be the category with objects the sets $[n] = \{0, 1, \ldots, n\}$, for $n \geq 0$. Morphisms among them are compositions of the non-decreasing maps: $d^i\colon [n-1]\Right0{} [n]$, $0\leq i \leq n$, defined $d^i(j)=j$ if $j<i$ and $d^i(j)=j+1$ if $j\geq i$, and $s^i\colon [n+1]\Right0{} [n]$ defined $s^i(j)=j$ if $j\leq i$ and $s^i(j)=j-1$ if $j> i$.

\begin{defi}\label{defisimplicial}

A \textit{simplicial set} is a functor $X\colon \Delta \op \longrightarrow \Sets$. Equivalently, a simplicial set is a sequence of sets $\{X_n = X([n])\}_{n \geq 0}$, together with structural maps
$$
d_i = X(d^i)\colon X_n\Right0{} X_{n-1}\,, \quad \mbox{and} \quad s_i=X(s^i)\colon X_n\Right0{} X_{n+1}\,
$$
called \emph{face maps} and \emph{degeneracies} respectively, satisfying the following \emph{simplicial identities}:
$$
\begin{array}{ll}
d_i d_j  = d_{j-1} d_i  &  i < j \\
d_i s_j  = s_{j-1} d_i &  i < j \\
d_j s_j = 1 = d_{j +1} s_j &  \\
d_i s_j = s_j d_{i-1} &  i > j + 1 \\
s_i s_j = s_{j+1} s_i &   i\leq j\,.
\end{array}
$$
The elements of $X_n$ are called the \emph{$n$-simplices} of $X$. Maps between simplicial sets are natural transformations $f\colon X\Right0{} Y$, in other words, collections of maps $f_n\colon X_n \Right0{}Y_n$ that commute with face maps and
degeneracies.

\end{defi}

\begin{expl}\label{nsimplex}

Let $n \geq 0$. Associated to $[n]$ there is a simplicial set $\Delta[n]$, called the \emph{standard $n$-simplex}, which is defined as the functor $\Hom_{\Delta}(-, [n]) \colon \Delta^{\op} \to \Sets$, see \cite[Example I.1.7]{GJ}. Equivalently, $\Delta[n]$ is characterized by the existence of a distinguished $n$-simplex, denoted by $\iota_n$ (corresponding to the identity map $[n] \to [n]$), and by the fact that every other simplex in $\Delta[n]$ is obtained from $\iota_n$ by applying to it a composition of face and degeneracy operations.

For instance,  the standard $1$-simplex $\Delta[1]$ plays a crucial role in the definition of homotopy below. It has a unique non-degenerate $1$-simplex $\iota$ and two vertices $v_0$ and $v_1$ obtained as 
$v_1=d_0(\iota)$ and $v_0=d_1(\iota)$. It has the homotopy type of the unit interval. The vertices can be represented by simplicial maps
$$
v_0 \colon \Delta[0]\Right0{}\Delta[1] \qquad \mbox{and} \qquad v_1\colon \Delta[0]\Right0{}\Delta[1].
$$

More generally, given a simplicial set $X$ and $n \geq 0$, one can represent each $n$-simplex $\omega \in X_n$ as the simplicial map $\omega \colon \Delta[n] \to X$, which maps the unique non-degenerate $n$-simplex $\iota_n$ of $\Delta[n]$ to the $n$-simplex $\omega \in X_n$ (note that this completely determines the map $\omega$). We shall tacitly use this fact throughout the rest of the paper.

\end{expl}

\begin{expl}\label{cartprod}

Let $X$ and $Y$ be simplicial sets. The \emph{cartesian product of $X$ and $Y$}, denoted by $X \times Y$, is the simplicial set with $n$-simplices $(X \times Y)_n = X_n \times Y_n$. Face and degeneracies on $X \times Y$ are defined componentwise. That is, given $(\omega, \sigma) \in (X \times Y)_n$ and $0 \leq i \leq n$, we define $d_i(\omega, \sigma) = (d_i^X(\omega), d_i^Y(\sigma))$ and $s_i(\omega, \sigma) = (s_i^X(\omega), s_i^Y(\sigma))$.

\end{expl}

\begin{defi}

Let $X$ and $Y$ be simplicial sets, and let $f, g \colon Y \to X$ be simplicial maps. A \emph{homotopy from $g$ to $f$} is a map
$$
F\colon Y\times \Delta[1] \Right2{} X
$$
such that $f= F\circ (\Id\times v_0)$ and $g= F\circ (\Id\times v_1)$ (cf.~\cite[pag.\ 24]{GJ}).

\end{defi}

The following operators are of special interest to us. Some definitions are borrowed from \cite[Section 6]{BLO6}.

\begin{defi}

Let $X$ be a simplicial set. For each $1 \leq i \leq n$, let $e_i^n\colon X\Right0{} X_1$ be the \emph{$i$th edge map} induced by the non-decreasing function $[1]\Right0{} [n]$ with image ${i-1,i}$.
\begin{itemize}

\item The \emph{spine operator} $\7e^n$, $n\geq1$, is defined as
$$
\7e^n  = (e^n_1,e^n_2,\dots, e^n_n) \colon X_n \Right4{}  X_1 \times \ldots \times X_1  \,. 
$$
\item The \emph{back edge} or \emph{product operator} $\Pi^n$, $n\geq1$,  is the map
$$
\Pi^n   \colon X_n \Right4{}  X_1 \,,  \qquad n\geq1\,,
$$
induced by the function $ [1] \Right0{} [n]$ with image $\{0,n\}$. 
Equivalently, $\Pi^1 = \Id$ and $\Pi^n= d_1\circ d_1\circ \dots\circ d_1$, $n-1$ times. We   simply write $\Pi = \Pi^n$ when there is no possible
confusion about the degree $n$.
\end{itemize}

\end{defi}

The following example illustrates and motivates the above definitions.

\begin{expl}\label{BG}
The bar construction on a group $G$, denoted by $BG$, is a simplicial set with $n$-simplices 
$B_nG=\Con{[g_1|g_2|\dots|g_n]}{g_i\in G} = G^n$, face maps
$$
 d_i([g_1| \dots | g_n]) = 
 \begin{cases}
[g_2| \ldots | g_n]                                       & \text{$i=0$, } \\
   [g_1| \ldots | g_i g_{i+1} | \ldots |  g_n ]  & \text{$1 \leq i \leq n-1$,} \\
[g_1| \ldots| g_{n-1}]                                   &   \text{$i=n$.}
\end{cases}
 $$
and degeneracies
$$
s_i([g_1| \ldots| g_n]) = [g_1| \ldots| g_i| 1|g_{i+1}|\ldots| g_n] \mbox{, } i = 0, \ldots, n\,.
$$
In this case, the edges operator is $\7e^n([g_1|g_2|\dots|g_n]) =(g_1,g_2,\dots,g_n)$, while the product operator is $\Pi ([g_1|g_2|\dots|g_n]) =g_1g_2\dots g_n$. Thus for each simplex, the product operator is precisely the product of its front edges. Notice that in this case the edges operator is actually a bijection for every $n\geq2$.
\end{expl}

The concepts of homotopy groups and homology groups are defined in the setting of simplicial sets. 
There is a realization functor that associates a topological space $|X|$ to a simplicial set $X$. In particular, 
homotopy groups and homology groups of $X$ coincide with that of $|X|$, defined geometrically. 
A simplicial map $f\colon X\Right0{} Y$ is a weak equivalence if it induces isomorphisms on homotopy groups.
Details can be easily find in the literature (cf.~\cite[Ch.~I]{GJ}\cite[Ch.I, II, and III]{May}).

\begin{expl}\label{nerve} The nerve of a small category $\calc$ is a simplicial set $N\calc$ with $n$-simplicies 
the  strings 
$ c_o\Left0{\alpha_1}c_1\Left0{\alpha_2}\dots \Left0{}  c_{n-1} \Left0{\alpha_n}c_n $
of $n$ composable  morphisms of $\calc$, face maps
$$
 d_i( c_o\Left0{\alpha_1}c_1\Left0{\alpha_2}\dots \Left0{}  c_{n-1} \Left0{\alpha_n}c_n ) =
\begin{cases}
   c_1\Left0{\alpha_2}\dots \Left0{}  c_{n-1} \Left0{\alpha_n}c_n    &  \text{$i=0$, } \\
   c_o\Left0{\alpha_1}   \dots  c_{i-1}\Left2{\alpha_i\alpha_{i+1}}c_{i+1}\dots\Left0{\alpha_n}c_n    & \text{$1 \leq i \leq n-1$,}\\
    c_o\Left0{\alpha_1}c_1\Left0{\alpha_2}\dots \Left0{\alpha_{n-1}}  c_{n-1}     &   \text{$i=n$,}
\end{cases}
$$
and degneracies 
$$s_i( c_o\Left0{\alpha_1}c_1\Left0{\alpha_2}\dots \Left0{}  c_{n-1} \Left0{\alpha_n}c_n)=
     c_o\Left0{\alpha_1} \dots  \Left0{\alpha_i} c_i \Left0{\Id} c_i  \Left0{\alpha_{i+1}} 
              \dots  c_{n-1} \Left0{\alpha_n}c_n\,.
$$
Usually, in the literature the simplicial set above is the nerve of the opposite category $\calc\op$ 
(cf.~\cite[Example I.1.4]{GJ}). 
It is well known that the nerve of a category and of its opposite cartegory are weak homotopy 
equivalent simplicial sets \cite[Section 1]{Qu2}. 
Our choice is made in consistency with the bar construction when the cartegory $\calc$ is the category $\calb G$
of a group $G$, namely, the category with one object and endomorphisms the group $G$. In fact, our definitions 
make $N(\calb G)\cong BG$ isomorphic simplicial sets.

\end{expl}


\section{Partial groups}\label{Spg}

We now introduce the concept of partial group, from a simplicial point of view. As we show, our definition is equivalent to Chermak's \cite[definition 2.1]{Chermak}.

\begin{defi}\label{def.part.mon}

A \emph{partial monoid} is a nonempty simplicial set  $\locm$ satisfying
\begin{enumerate}[\rm (PM1)]

\item $\locm$ is reduced ($\locm_0$ consists of a unique vertex). 

\item The spine operator $\7e^n \colon \locm_n\Right0{} (\locm_1)^n$ is injective for all $n\geq1$.

\end{enumerate}
A \emph{homomorphism of partial monoids} is a simplicial map.
\end{defi}

One can easily check that the above definition is equivalent to \cite[Definition 2.1]{Chermak}. Indeed, using Chermak's notation, we can regard the partial monoid $\locm$ as having underlying set $\calm=\locm_1$, the set of edges of $\locm$, and the domain of multiplication is given by
$$
\domd = \{(\emptyset)\}\cup\coprod_{n\geq1} \con{(x_1,\dots,x_n)\in(\locm_1)^n}{ \exists \sigma \in \locm_n,\,, \7e^n(\sigma)=
(x_1,\dots,x_n)}\,.
$$
The multiplication $\Pi$ corresponds to the product operator of $\locm$, and the 
unit is $1=s_0(v)\in \locm_1$, where $v$ stands for the unique vertex of $\locm$. 
Conversely, given a partial monoid defined as in \cite{Chermak}, the bar construction, 
defined exactly as in Example~\ref{BG}, is a simplicial set that satisfies conditions (PM1) and (PM2).
\smallskip

The above correspondence suggests the following notational conventions that we   use henceforth.
For a partial monoid $\locm$, a simplex $\7x\in\locm_n$ with spine 
$\7e^n(\7x)=(x_1,x_2, \dots, x_n)$  
is denoted
$$\7x = [x_1|x_2|\dots|x_n]\,,
$$
thus using the same notation as for the bar construction on a group $G$.
Concerning the multiplication, we   write $1=s_0(v)$, where $v$ is the vertex of $\locm$, and 
$$
\Pi [x_1|x_2|\dots|x_n] = x_1\cdot x_2\cdot \ldots \cdot x_n\in \locm_1\,.
$$
To be more precise, the symbol $ x_1\cdot x_2\cdot \ldots \cdot x_n$ means that the simplex 
$[x_1|x_2|\dots|x_n]$ exists in $\locm$ and 
$\Pi [x_1|x_2|\dots|x_n] = x_1\cdot x_2\cdot \ldots \cdot x_n$.

\begin{defi}\label{pgroup}

An \emph{inversion} in a partial monoid $\locm$  is an anti-involution $\nu\colon \locm\Right0{}\locm$ (that is, a simplicial map $\nu\colon\locm\Right0{}\locm\op$ with inverse $\nu^{-1} = \nu\op$), satisfying the following conditions for each $\7u \in \locm_n$, $n\geq1$:
\begin{enumerate}[\rm ({I}1)]

\item There is a simplex $[\nu(\7u)|	\7u] \in \locm_{2n}$; and

\item $\Pi[\nu(\7u)|\7u] = 1$.

\end{enumerate}
A \emph{partial group} is a partial monoid together with an inversion. A \emph{partial subgroup} is a simplicial subset $\loch\subseteq \locm$ satisfying the following conditions.
\begin{enumerate}

\item  If $\7x=[x_1|x_2|\dots|x_n]\in \locm_n$, and each $[x_i]$ is in $\loch_1$, then $\7x\in\loch_n$.

\item The inversion and the product of $\locm$ are internal in $\loch$.

\end{enumerate}
In this case, we denote $\loch\leq \locm$. If $\loch=BG$ is the classifying space of a group $G$, we   simply say that $G$ is a subgroup of $\locm$, and loosely write $G\leq \locm$, instead of $BG\leq \locm$.

\end{defi}

Continuing with our notational conventions, we won't ususally specify the map $\nu$ 
providing the inversion of a partial group, and instead, we write $\7u^{-1}=\nu(\7u)$. 
The simplicial identities imply that $[u_1|\ldots|u_n]^{-1} = [u_n^{-1}| \ldots | u_1^{-1}]$.

It follows from simplicial identities that if a partial monoid $\locm$ admits an inversion, then this is unique. 
Also, a homomorphism of partial groups $f \colon \locm \to \loch$ satisfies 
$f(1_{\locm}) = 1_{\loch}$ and $f(\7w^{-1}) = (f(\7w))^{-1}$. These properties are also shown in 
\cite[2.2, 2.3]{Chermak}.

\begin{rmk}
Partial groups are not generally Kan complexes, that is, fibrant objects in the Quillen model category of simplicial sets.  
It turns out that a partial group is really a group  if and only if the spine operators are bijections 
for all $n\geq1$, in which case it is a Kan complex. (See \cite[1\S1]{May}\cite[I.3]{GJ}.)
As a consequence, the homotopy relation between homomorphisms fails to be transitive in general.
\end{rmk}

%
%
%
%
%
%
%
%


We start here a careful analysis of homotopies between homomorphisms of partial groups and  their relation with the 
concept of conjugation. This leads at the end of this section to the definition of the groups of automorphisms and 
outer automorphisms of partial groups.

\begin{lmm}\label{maps}

Let $f\colon X \Right0{} \locm$ be a simplicial map where $\locm$ is a partial monoid. Then, $f$ is determined by its restriction to the $1$-skeleton of $X$.

\end{lmm}

\begin{proof}
Let $g\colon X\Right0{}\locm $ be another simplicial map with $g_1 = f_1$. 
Since there is a unique vertex in $\locm$, it is also clear that $g_0 = f_0$. 
Assume $n\geq2$ and let $\omega \in X_n$ be an arbitrary $n$-simplex with 
$\7e(\omega) = (x_1, \ldots, x_n) \in X_1 \times \ldots \times X_1$. 
Since both $f$ and $g$ are simplicial, we have
$$
\7e^{n}(f_n(\omega)) = 
(f_1(x_1), \ldots, f_1(x_n))   = (g_1(x_1), \ldots, g_1(x_n)) = \7e^{n}(g_n(\omega))\,,
$$
hence $f_n(\omega) = g_n(\omega)$ by (PM2), and $f=g$.
\end{proof}

%
%
%
%
%
%
%

\begin{lmm}\label{homot=conj}

Let  $f,g \colon \loch\Right0{}\locm$ be two homomorphisms of partial monoids. Then, the following holds.
\begin{enumerate}[\rm(a)]

\item A homotopy $F\colon \loch\times \Delta[1]\Right0{}\locm$ from $g$ to $f$ is uniquely determined by the $1$-simplex $\eta=F(1,\iota)\in \locm_1$, that satisfies $\eta\cdot g(x) = f(x)\cdot \eta$, for all $x\in \loch_1$.

\item A simplex $\eta\in\locm_1$ determines a homotopy from $g$ to $f$ if and only if,
for each simplex $[x_1|\dots|x_n]\in \loch$, the following conditions are satisfied:

\begin{enumerate}[\rm (i)]

\item $\7w_k = [f(x_1)|\dots|f(x_{k})|\eta|g(x_{k+1})|\dots |g(x_n)]\in \locm$ for each $k = 0, \ldots, n$; and

\item $\Pi(\7w_0) = \Pi(\7w_1) = \ldots = \Pi(\7w_n)$.

\end{enumerate}

\end{enumerate}
In addition, if $\locm$ and $\loch$ are partial groups, then the following holds.
\begin{enumerate}[\rm(a)]\setcounter{enumi}{2}

\item A homotopy $F\colon \loch\times \Delta[1]\Right0{}\locm$ is uniquely determined by one of the functions, $F_0 = F|_{\loch \times \{v_0\}}$ or $F_1 = F|_{\loch \times \{v_1\}}$, together with $\eta=F(1,\iota)\in\locm_1$.

\end{enumerate}

\end{lmm}

\begin{proof}

By Examples \ref{nsimplex} and \ref{cartprod}, we know that the non-degenerate $1$-simplices of $\loch\times \Delta[1]$ are of one of the forms below:
$$
(1,\iota)\qquad(x, \iota)\qquad (x, s_0(v_0))\qquad (x, s_0(v_1))
$$
where $x$ runs over all non-degenerate 1-simplices of $\loch$.
According to Lemma~\ref{maps}, the map $F$ is determined by the images of these $1$-simplices, that is 
$
F(1,\iota)=\eta$, 
$F(x, s_0(v_0)) = f(x)$, 
$F(x, s_0(v_1))  = g(x)$, and $F(x, \iota)$.
We only need to check that this last  is also determined by $\eta$, $f(x)$ and $g(x)$.

Notice that each $1$-simplex $x \in \loch_1$ determines two non-degenerate $2$-simplices of  $\loch\times\Delta[1]$, namely $\sigma_1=(s_0(x),s_1(\iota))$ and $\sigma_2=(s_1(x),s_0(\iota))$.
The edges of $F(\sigma_1)$ and of $F(\sigma_2)$ are computed using the simplicial identities:
$$
d_0(F(\sigma_1)) = F(d_0(\sigma_1)) = F\bigl(  (d_0s_0(x),d_0s_1(\iota))  \bigr) = F\bigl( (x,s_0(v_1))\bigr) = F_1(x)=g(x).
$$
Similarly, we have $d_0(F(\sigma_2)) = \eta$, and there are identities
$$
\begin{array}{ccc}
d_1(F(\sigma_1))    = F(x,\iota) & \qquad & d_2(F(\sigma_1)) =  \eta\\
d_1(F(\sigma_2))   = F(x,\iota) & & d_2(F(\sigma_2))    = F_0(x) =f(x) \\
\end{array}
$$
Hence, we deduce the following equalities (see Figure~\ref{pic1} below):
$$
F(\sigma_1) = [\eta | F_1(x)] = [\eta | g(x)] \qquad \mbox{and} \qquad F(\sigma_2) = [F_0(x)|\eta ] = [f(x)|\eta].
$$
and then 
\begin{equation*}
 \eta\cdot g(x)  = d_1[\eta | g(x)]  = F(x,\iota) = d_1[f(x)|\eta] = f(x)\cdot \eta\,.
\end{equation*}
This shows that $F(x,\iota)$ is also determined by $\eta$, $f(x)$ and $g(x)$. This proves (a). 
\medskip

\begin{figure}[h]
 \setlength{\unitlength}{0.7mm}
 \begin{picture}(50, 50)(0,0)
\put(5,5){\circle*{2}}
\put(45,5){\circle*{2}}
\put(5,45){\circle*{2}}
\put(45,45){\circle*{2}}
\put(5,5){\line(0,1){40}}
\put(5,5){\line(1,0){40}}
\put(5,5){\line(1,1){40}}
\put(5,45){\line(1,0){40}}
\put(45,5){\line(0,1){40}}
\put(30,15){\tiny$\sigma_1$}
\put(13,35){\tiny$\sigma_2$}
\put(-4,0){\tiny$(v,v_0)$}
\put(43,0){\tiny$(v,v_1)$}
\put(-4,48){\tiny$(v,v_0)$}
\put(43,48){\tiny$(v,v_1)$}
\put(21,2){\tiny$(1,\iota)$}
\put(21,47){\tiny$(1,\iota)$}
\put(-16,25){\tiny$(x,s_0(v_0))$}
\put(47,25){\tiny$(x, s_0(v_1))$}
\put(16,20){\rotatebox{45}{\tiny$(x,\iota)$}}
\end{picture}
\hspace{8em}
%
\begin{picture}(50, 50)(0,0)

\put(5,5){\circle*{2}}
\put(45,5){\circle*{2}}
\put(5,45){\circle*{2}}
\put(45,45){\circle*{2}}

\put(5,5){\line(0,1){40}}
\put(5,5){\line(1,0){40}}
\put(5,5){\line(1,1){40}}
\put(5,45){\line(1,0){40}}
\put(45,5){\line(0,1){40}}

\put(30,15){\tiny$F(\sigma_1)$}
\put(13,35){\tiny$F(\sigma_2)$}

\put(2,0){\tiny$v$}
\put(45,0){\tiny$v$}
\put(2,48){\tiny$v$}
\put(45,48){\tiny$v$}
\put(23,2){\tiny$\eta$}
\put(23,47){\tiny$\eta$}
\put(-5,25){\tiny$f(x)$}
\put(47,25){\tiny$ g(x)$}
\put(16,20){\rotatebox{45}{\tiny$F(x,\iota)$}}

\end{picture}
\caption{Picture of the simplices in 
$(x,\iota)\colon \Delta[1]\times \Delta[1]\subset \loch\times \Delta[1]$ on the left and the image by 
a homotopy $F\colon \loch\times \Delta[1]\Right0{}\locm$ on the right hand side. (We write $v$ for both the vertex of $\loch$ and of $\locm$.)}\label{pic1}
\end{figure}

The proof of part (b) follows from the characterization of simplicial homotopies described in 
\cite[Definition I.5.1]{May}. Assuming the existence of the simplices $\7w_k$ in condition 
(i), we can define functions $h_k\colon \loch_n\Right0{}\locm_{n+1}$ such that 
$$h_k(\7x) = \7w_k =  [f(x_1)|\dots|f(x_{k})|\eta|g(x_{k+1})|\dots |g(x_n)]\in \locm$$
 for each $\7x=[x_1|\dots|x_n]\in \loch$ and $k = 0, \ldots, n$. Using condition (ii),
one can easily check the list of relations in  \cite[Definition I.5.1]{May}. Therefore these
functions define a homotopy from $g$ to $f$, with $F(1,\iota)=\eta$.

Conversely, if $F$ is a homotopy between  $f$ and $g$, with 
$F(1,\iota)=\eta\in\locm_1$, the generalization of the arguments of part (a) 
to higher dimensions
gives the desired results. Indeed, a simplex 
$\7x=[x_1|x_2|\dots|x_{n}]\in\loch_n$, $n\geq2$,  determines a prism 
$\Delta[n]\times \Delta[1]$ in  $\loch\times\Delta[1]$ with non-degenerate
$(n+1)$-simplices $\4{\7x}_k = \bigl(\, s_k(\7x), s_{k+1}^{n-k}s_0^k(\iota) \,\bigr)$,
 $k=0,\dots,n$. Then we set
$$
\7w_k = 
F(\4{\7x}_k) = [f(x_1)|\dots|f(x_{k})|\eta |g(x_{k+1})|\dots|g(x_{n}) ]  \in \locm_{n+1} \,, \quad k=0,\dots,n.
$$
Again, using the simplicial identities listed in Definition \ref{defisimplicial}, we obtain the following  equality $\Pi(\4{\7x}_k) =(d_1)^n \bigl(\, s_k(\7x), s_{k+1}^{n-k}s_0^k(\iota) \,\bigr) = (\Pi(\7x), \iota)\in (\loch\times \Delta[1])_1$, and hence 
$$
\Pi(F(\4{\7x}_k) ) = F(\Pi(\7x), \iota) = f(x_1\cdot x_2\cdot\ldots\cdot x_n)\cdot\eta= 
\eta \cdot g(x_1\cdot x_2\cdot\ldots\cdot x_n)
$$
for all $k=0,\dots, n$. Therefore,  the conditions (i) and (ii) of (b) are satisfied.
\medskip

Part (c) follows from the computations in the proof of part (a). Suppose in addition that $\locm$ and $\loch$ are partial groups and assume $F$ and $G$ are homotopies with $F_0 = G_0$, and $F(1, \iota) = \eta = G(1, \iota) \in \locm_1$. Then, \ref{homot=conj}(a) applies to both, $F$ and $G$, to give 
$$
\eta \cdot F_1(x) =  F_0(x)\cdot \eta = G_0(x)\cdot\eta = \eta\cdot G_1(x)\,.
$$
But $\locm$ is a partial group, hence the above equality implies that $F_1(x) = G_1(x)$. 
 Since this holds for all non-degenerate $x\in \loch_1$ as it also clearly holds for $x=1$, 
 Lemma~\ref{maps} implies that $F_1= G_1$. Thus, $F$ is determined by only $F_0$ and 
 $\eta= F(1,\iota)$. A similar argument implies that $F$ is also determined by $F_1$ and $\eta$.
\end{proof}

\begin{nota} Given two homomorphisms of partial groups, $f, g \colon \loch \to \locm$, 
the symbol 
$$
f \Left2{\eta} g
$$
stands for a homotopy from $g$ to $f$ which is determined by  $\eta \in \locm_1$, according
to Lemma~\ref{homot=conj}. Furthermore, since according to Lemma~\ref{homot=conj}(c) 
$\eta$ and $f$ determine $g$, we   write
$f= \9\eta g$. Likewise  $f^\eta=g$. 
\end{nota}

\begin{lmm}\label{ho.eq}

Let  $f,g \colon \loch\Right0{}\locm$ be two homomorphisms of partial groups and assume that 
$\eta \in \locm_1$ determines a homotopy from $f$ to $g$: $f \Left2{\eta} g$. 
Then, the following holds.
\begin{enumerate}[\rm (a)]

\item $\eta^{-1}$ determines a homotopy $f\Right1{\eta^{-1}}g$.

\item If $j\colon \lock\Right0{}\loch$ is  
homomorphism of partial groups, then  $\eta \in \locm_1$ determines a 
homotopy  $f\circ j \Left1{\eta} g\circ j$.

\item If $k\colon \locm\Right0{}\locl$ is a homomorphism of partial groups, 
then $k(\eta)$ determines a homotopy $k\circ f\Left1{k(\eta)}k\circ g$.

\item If $i,j\colon \lock\Right0{}\loch$ are homomorphisms of partial groups and 
$\nu\in \loch_1$ determines a homotopy $i\Left1{\nu}j$, then 
$\eta\cdot g(\nu)=f(\nu)\cdot\eta\in \locm_1$
determines a homotopy 
$f\circ i\LEFT6{\eta\cdot g(\nu)=\ }{\ = f(\nu)\cdot\eta}g\circ j$.
\end{enumerate}
%

\end{lmm}

\begin{proof}
For each  
$\7h = [h_1|\ldots|h_n] \in \loch$, we have  $\7h^{-1} = [h_n^{-1}|\ldots|h_1^{-1}] \in \loch$. Since $\eta$ defines a homotopy from $g$ to $f$, conditions (i) and (ii) in 
Lemma~\ref{homot=conj} (b) applied 
to $\7h^{-1}$ provide  simplices
$$
\7v_k = [f(h_n^{-1})|\ldots |f(h_{n-k+1}^{-1})|\eta|g(h_{n-k}^{-1})|\ldots|g(h_1^{-1})] 
       \in \locm\,,\qquad k = 0, \ldots, n\,,
$$
with $\Pi(\7v_0) = \Pi(\7v_1) = \ldots = \Pi(\7v_n)$. Inverting again, we get
$$
\7u_{n-k} = \7v_k^{-1} = [g(h_1)|\ldots|g(h_{n-k})|\eta^{-1}|f(h_{n-k+1})|\ldots|f(h_n)] \in \locm\,,\qquad k = 0, \ldots, n\,,
$$
with $\Pi(\7u_0) = 
\ldots = \Pi(\7u_n)$. Hence, by Lemma~\ref{homot=conj}(b), $\eta^{-1}$ defines a homotopy from $f$ to $g$, and this proves (a). The proofs of  (b), (c), and (d)  follow by similar arguments based  on Lemma~\ref{homot=conj}(b)
\end{proof}

\begin{defi}\label{definorm}

The \emph{normalizer} $N(\locm)$ of a partial group  $\locm$ is the set of elements 
$\eta \in \locm_1$ that define a homotopy of the identity on $\locm$: 
$$
    N(\locm) = \conj{\eta \in \locm_1}{\exists\ \cj{\eta} \Left1{\eta}1_\locm}\,.
$$
The \textit{center} of $\locm$ is the subset $Z(\locm) \subseteq N(\locm)$ of elements $\eta$ defining a self-homotopy of the identity ($\cj{\eta} =1_\locm$).

\end{defi}

For any $\eta\in N(\locm)$, the map $\cj\eta$ can be seen as a standard conjugation map and it shares many properties with conjugation maps in honest groups. 
Indeed, $\cj\eta$ is a well defined automorphism, characterized by the formula $\cj\eta(x)=\eta\cdot x\cdot\eta^{-1}$, for each $x\in\locm_1$. 
This is shown in Lemma~\ref{N(M)conj} below. It turns out that $N(\locm)$ is a group and 
$\cj{}$ defines a group homomorphism $\cj{}\colon N(\locm)\Right0{}\Aut(\locm)$. By analogy with the case of 
groups, the automorphisms in the image of $\cj{}$ are called inner, and the cokernel is defined to be the 
outer automorphsims. Proposition \ref{Nloclgroup} below formalizes this idea. 

Notice that if $\locm=BG$ for some group $G$, then $N(BG)= G$. In general though, 
we can only guarantee that the unit belongs to $N(\locm)$.

\begin{lmm}\label{N(M)conj}
 Let $\locm$ be a partial group and let $\eta\in\locm_1$. Then, the following holds.
 
\begin{enumerate}[\rm (a)]
\item\label{N(M)conj1} $\eta\in N(\locm)$ if and only if 
for each simplex 
$[x_1|\dots |x_n]\in\locm$ there are also simplices 
$[ x_1|\dots|x_k|\eta| \eta^{-1}|x_{k+1}| \ldots |x_n]\in \locm$, for $k=0,\dots,n$.
\smallskip

\item\label{N(M)conj2} If $\eta\in N(\locm)$, then 
for each $x\in\locm_1$, we have $[\eta|x|\eta^{-1}]\in \locm$ and 
$$
    \cj\eta(x)  = \eta\cdot x\cdot \eta^{-1}\,.
$$

\item $[\eta_1| \ldots | \eta_n] \in \locm$ for any sequence 
$\eta_1, \ldots, \eta_n \in N(\locm)$.
\label{Normaux(a)}
\smallskip

\item If $\eta_1, \eta_2 \in N(\locm)$, then $\eta = \eta_1 \cdot \eta_2 \in N(\locm)$, 
and   $\cj{\eta} = \cj{\eta_1}\circ \cj{\eta_2}$.
\label{Normaux(b)}
\smallskip

\item If $\eta \in N(\locm)$ then $\eta^{-1} \in N(\locm)$, 
and $\cj{\eta^{-1}}= (\cj{\eta})^{-1}$.
\label{Normaux(c)}

\end{enumerate}  
 
\end{lmm}

\begin{proof}
Given $\cj{\eta}\Left1{\eta} 1_\locm$, 
we also have $1_\locm \Left1{\eta^{-1}}  \cj{\eta}$ according to 
Lemma~\ref{ho.eq}(a). According to Lemma~\ref{ho.eq}(b), composing with the inverse map 
of $\cj{\eta}$, 
we get 
\begin{equation}\label{inv-homot}
 (\cj{\eta})^{-1}  \Left2{\eta^{-1}}   1_\locm\,,
\end{equation}
thus $\eta^{-1}\in N(\locm)$ and $(\cj{\eta})^{-1} =\cj{\eta^{-1}}$. This proves \eqref{Normaux(c)}.

Assume that $\eta\in N(\locm)$. According to Lemma~\ref{homot=conj}(b), the existence of the 
homotopy \eqref{inv-homot} implies the existence of simplices
$\7z_k = [\cj{\eta}^{-1}( x_1)|\dots|\cj{\eta}^{-1}( x_k)|\eta^{-1}|x_{k+1}| \ldots |x_n]\in \locm$, for $k=0,\dots,n$. Applying Lemma~\ref{homot=conj}(b) again to the homotopy  $\cj{\eta} \Left1{\eta}   1_\locm$ and the simplices $\7z_k$, we obtain simplices 
$$
    \7y_k=  [ x_1|\dots|x_k|\eta| \eta^{-1}|x_{k+1}| \ldots |x_n]\in \locm
$$
for every $k=0,\dots,n$.

Conversely, let $\eta \in \locm_1$. For each simplex $\7x = [x_1|\dots |x_n]\in \locm$, assume also that the simplices $\7y_k$ described above exist for all $k$. Applying this property recursively, we get new simplices 
$
\7v_k =   [\eta|\eta^{-1}|x_1|\eta|\eta^{-1}|x_2|  \dots  |\eta|\eta^{-1}  |x_k|\eta|\eta^{-1}| x_{k+1}|x_{k+2}\dots ]
$,
for $k=0,\dots, n$, and they clearly satisfy $\Pi(\7v_0) =\dots =  \Pi(\7v_n) =\Pi(\7x)$. Applying the product operator appropriately we obtain simplices
$$
\7w_k =   [\eta^{-1}\cdot x_1\cdot \eta |  \dots 
                |\eta^{-1}\cdot x_k\cdot \eta|\eta^{-1}| x_{k+1}|x_{k+2}\dots ]\,,
$$
for $k=0,\dots, n$, such that $\Pi(\7w_0) =\dots =  \Pi(\7w_n)$. By Lemma \ref{homot=conj}, this proves that $\eta^{-1} \in N(\locm)$, and by part (e) it now follows that $\eta \in N(\locm)$. This finishes the proof of \eqref{N(M)conj1}.

Notice that the above also proves part \eqref{N(M)conj2}. Indeed, the simplices $[\eta^{-1}|x|\eta]$ appear as faces of $\7v_k$, and the simplices $\7w_k$ show that we can write 
$\cj{\eta^{-1}}(x) = \eta^{-1} \cdot x \cdot \eta$, by Lemma~\ref{homot=conj}(c). Note that we can also switch the roles of $\eta^{-1}$ and $\eta$ in order to obtain the identity $\cj\eta(x) = \eta\cdot x \cdot \eta^{-1}$. Finally, \eqref{Normaux(a)} and \eqref{Normaux(b)} follow immediately from \eqref{N(M)conj1} and 
\eqref{N(M)conj2}. 
\end{proof}

\begin{nota}\label{not-conj}
Given a partial group $\locm$ and $\eta\in N(\locm)$, we shall write
$$
\cj\eta(x) = \9\eta x=\eta\cdot x \cdot \eta^{-1} \qquad \mbox{and} \qquad (\cj{\eta})^{-1}(x) = \cj{\eta^{-1}}(x) = x^\eta =\eta^{-1}\cdot x \cdot \eta.
$$
\end{nota}

 \begin{prop}\label{Nloclgroup}

Let $\locm$ be  a partial group. Then, $N(\locm)$ is a subgroup of $\locm$, $Z(\locm)$ is an abelian subgroup of $N(\locm)$, and there 
is an exact sequence
\begin{equation}\label{eqexactseq1}
1 \to Z(\locm) \Right3{} N(\locm) \Right3{\cj{}} \Aut(\locm) \Right3{} \Out(\locm) \to 1,
\end{equation}
where $\cj{}\colon N(\locm) \to \Aut(\locm)$ is the map that sends $\eta \in N(\locm)$ 
to $\cj{\eta} \in \Aut(\locm)$ and $\Out(\locm)$ is the set of homotopy classes of 
automorphisms of $\locm$. In particular, $N(\locm)$ and $Z(\locm)$ are invariant 
by automorphisms of $\locm$, and $\Out(\locm)$ is a group.

\end{prop}

\begin{proof}

Notice first that $N(\locm)$ is a group by Lemma~\ref{N(M)conj}(\ref{Normaux(a)},\ref{Normaux(b)},\ref{Normaux(c)}), as well as a subgroup of $\locm$. The assignment $\eta \mapsto \cj{\eta}$ defines a homomorphism 
$N(\locm) \to \Aut(\locm)$ by Lemma~\ref{N(M)conj}\eqref{Normaux(b)}.  By definition, $Z(\locm)$ is the kernel of this homomorphism, and it is clearly abelian.     

We can now show that the homotopy relation between automorphisms of $\Aut(\locm)$ is really an equivalence relation. 
According to Lemma~\ref{homot=conj}(a) a homotopy from $f$ to $g$
is determined by an element $\eta\in\locm_1$, $g\Left1{\eta}f$. 
By Lemma \ref{ho.eq} (b), we also have $g\circ f^{-1} \Left1{\eta}f\circ f^{-1} = \Id_\locm$, 
hence $\eta\in N(\locm)$ by definition. 
Moreover, $g\circ f^{-1}$ is determined by $\eta$ by Lemma~\ref{homot=conj}(c) and we write  
$g\circ f^{-1} = \cj\eta$ according to Definition~\ref{definorm}. 
Now it is clear that $\sim$ is an equivalence relation, by Lemma~\ref{N(M)conj}(\ref{Normaux(a)},\ref{Normaux(b)},\ref{Normaux(c)}).

Notice also that $f\circ\cj{\eta} = \cj{f(\eta)}\circ f$, hence the natural projection 
$\Aut(\locm)\Right0{}\Out(\locm)$ is a group epimorphism with kernel $\cj{}(N(\locm))$. 
This proves that the sequence~(\ref{eqexactseq1}) is exact. 
%
%
%
%
%
%
%
%
%
\end{proof}

The above result also shows that every automorphism of $\locm$ preserves the subgroup $N(\locm)$, and in this sense we say that $N(\locm)$ is a \emph{characteristic} subgroup of $\locm$. However, this does not imply that $N(\locm)$ is \emph{normal} in $\locm$.

\begin{lmm}\label{he=iso}
 A homotopy equivalence $f\colon\locm\Right0{}\locl$  of partials groups is an isomorphism.
\end{lmm}
\begin{proof}
 Let $g\colon \locl\Right0{}\locm$ be the homotopy inverse of $f$. Then $f\circ g$ is homotopic to the 
 identiy on $\locl$, hence there is $\eta\in N(\locl)$ such that $g\circ f= \cj{\eta}$. This last is an isomorphism
by Lemma~\ref{N(M)conj}(\ref{Normaux(c)}), hence $g$ is injective and $f$ is surjective. The same argument 
applied to $g\circ f$ shows that both $f$ and $g$ are isomorphisms.
\end{proof}


\section{The space of self-equivalences of a partial group}\label{Sself}

In this section we analyze in detail the structure of the simplicial group of automorphisms of a partial group.
By Lemma~\ref{he=iso},  we identify the space of self-equivalences of 
$\locm$ with the simplicial group of automorphisms.  
We show that this complex can be completely described in terms of the algebraic structure of the partial group. We start with a quick review on the definition of the automorphism complex of a simplicial set, which we then specialize to partial groups.

Given simplicial sets $X$ and $Y$, the \emph{function complex} $\Map(X,Y)$ is defined  as the simplicial set with $n$-simplices
$$
\Map(X,Y)_n = \{\mbox{simplicial maps } f \colon X \times \Delta[n] \to Y\}.
$$
and with face maps and degeneracies induced by face maps and degeneracies  the standard simplices $\Delta[n]$.
 (cf.\ \cite[I.5]{GJ}, \cite[6.3]{Curtis}). 

 In particular, when $X = Y$, the function complex $\Map(X,X)$ is naturally a simplicial monoid. 
 Each $n$-simplex $f\colon X\times \Delta[n]\Right0{} X$ determines and  is determined by the map
 $\widetilde{f} = (f, pr_2) \colon X \times \Delta[n] \to X \times \Delta[n]$.
 Then,  the product of two
 $n$-simplices $f$ and $g$ is defined as the composition 
 $\widetilde{f}\circ g$:
$$
\xymatrix{
X \times \Delta[n] \ar[r]^-{\widetilde{f}} \ar@/_3ex/[rr]_{fg} &  
          X \times \Delta[n] \ar[r]^-{{g}}  & X \,.
}
$$
 This makes $\Map(X,X)_n$ into a monoid for each $n \geq 0$, and the operation is easily checked to commute with face and degeneracy operators, so that $\Map(X,X)$ becomes a simplicial monoid. 

\begin{defi}

Let $X$ be a simplicial set. The \emph{automorphism complex of $X$} is the subcomplex $\aut(X) \subseteq \Map(X,X)$ of invertible elements (in every dimension). It is  a simplicial group by definition.

\end{defi}

Let $\locm$ be a partial group. In this case $\aut(\locm)$ coincides with the union of the connected components of 
$\Map(\locm,\locm)$ containing the self-equivalences. Indeed, the $0$-simplices are the  invertible simplicial maps
$$
\locm \cong \locm \times \Delta[0] \Right3{} \locm\,,
$$
$\aut_0(\locm) =\Aut(\locm)$. Moreover, for $n \geq 1$, each simplicial map
$$
\locm \times \Delta[n] \Right3{} \locm
$$
is determined by the restriction to the 1-skeleton of $\locm\times \Delta[n]$, as pointed out in Lemma \ref{maps}, and this fact allows a simplified description of $\aut(\locm)$.

Let $\calaut(\locm)$ be the category with object set the group $\Aut(\locm)$ and with morphism sets
$$
\Mor_{\calaut(\locm)}(\alpha, \beta) = \conj{ \eta \in N(\locm) }{\alpha\Left2{\eta} \beta}\,.
$$
That is, $\Mor_{\calaut(\locm)}(\alpha, \beta)$ is the set of elements of $N(\locm)$ that define a homotopy from $\beta$ to $\alpha$, namely $(\alpha \Left2{\eta} \beta)$. This category is strict monoidal with product $\otimes$ defined on objects as composition of automorphisms, 
$\alpha\otimes \beta= \alpha\circ\beta$, and on morphisms as
\begin{equation}\label{tensor}
\begin{split}
(\alpha\Left2{\eta} \beta)\otimes
(\alpha'\Left2{\eta'} \beta') &= (\alpha\otimes\alpha' \Left3{\ \eta\otimes\eta'} \beta\otimes\beta')\,, \\
\eta\otimes\eta' \defin \eta\cdot \beta(\eta') & = \alpha(\eta')\cdot \eta \in N(\locm)
\end{split}
\end{equation}
The rightmost equality above follows immediately from Lemma \ref{homot=conj} (b), as $\eta$ defines a homotopy from $\beta$ to $\alpha$.

The nerve of $\calaut(\locm)$ is therefore a simplicial group and it acts naturally on $\locm$ via the following formula
$$
(\alpha_0\Left1{\eta_1} \alpha_1 \Left1{\eta_2}\dots \Left1{\eta_n}\alpha_n) \cdot [x_1|\dots | x_n] 
= [\eta_1\alpha_1(x_1)|\dots | \eta_n\alpha_n(x_n)] \,,
$$
where the latter is a well defined simplex of $\locm$  by Lemma~\ref{homot=conj}, 
since $\eta_i\in N(\locm)$ for all $i=1,\dots,n$.

Let us illustrate with the case $n = 1$ how this is a well defined action. Fix $[x] \in \locm$, and notice that $(\Id \Left2{1} \Id)$ acts trivially on $[x]$. Let now $(\alpha_0 \Left2{\eta} \alpha_1), (\beta_0 \Left2{\gamma} \beta_1) \in \Mor(\calaut(\locm))$. On the one hand, we have
\begin{align*}
(\beta_0 \Left2{\gamma} \beta_1)[x] & = [\gamma \cdot \beta_1(x)] \\
(\alpha_0 \Left2{\eta} \alpha_1)[\gamma \cdot \beta_1(x)] & 
= [\eta \cdot \alpha_1(\gamma) \cdot (\alpha_1 \circ \beta_1)(x)]
\end{align*}
while, on the other hand, equation \eqref{tensor} above implies that
\begin{align*}
(\alpha_0 \Left2{\eta} \alpha_1) & \otimes (\beta_0 \Left2{\gamma} \beta_1)[x] = \\
 & = (\alpha_0 \circ \beta_0 \Left6{\eta \cdot \alpha_1(\gamma)} \alpha_1 \circ \beta_1)[x] = [\eta \cdot \alpha_1(\gamma) \cdot (\alpha_1 \circ \beta_1)(x)].
\end{align*}
The general case is similarly checked.

\begin{thm}\label{isoaut}

The adjoint map of the natural action  $\theta \colon |\calaut(\locm)| \times \locm \to \locm$ is 
an isomorphism of simplicial groups  
$$\omega\colon \nv{\calaut(\locm)} \Right2{\cong} \aut(\locm)\,.$$ 
%
%
\end{thm}

\begin{proof}

It follows because elements of $\aut(\locm)$ are determined by the restriction to the $1$-skeleton, and therefore they are in the image of $\omega$. As $\omega$ is also injective, this finishes the proof.
\end{proof}

\begin{cor}\label{isoaut2}

There is a fibration $B^2Z(\locm) \to B\aut(\locm) \to B\Out(\locm)$.

\end{cor}

\begin{proof}

By Theorem \ref{isoaut}, $\aut(\locm) \cong |\calaut(\locm)|$. The description of $\calaut(\locm)$ easily yields that $\pi_0(\aut(\locm)) \cong \Out(\locm)$ and each connected component of $\aut(\locm)$ is equivalent to $BZ(\locm)$. The statement follows easily.
\end{proof}

\begin{defi}\label{defiaction}

Let $\loch$ and $\locm$ be partial groups.
\begin{itemize}

\item An \emph{action} of $\loch$ on $\locm$ is a homomorphism of partial groups $\loch \to B\Aut(\locm)$.

\item A \emph{proxy action} of $\loch$ on $\locm$ is a simplicial map $\loch \to B\aut(\locm)$.

\item An \emph{outer action} of $\loch$ on $\locm$ is a homomorphism of partial groups $\loch \to B\Out(\locm)$.

\end{itemize}

\end{defi}



\section{Fibre bundles and extensions}\label{Sfibund}

In this section we review the concept of fibre bundle for simplicial sets, and study certain features of fibre bundles involving partial groups. As we show in this sextion, fibre bundles provide the right setup to develop a consistent extension theory for partial groups. The theory of fibre bundles of simplicial sets is widely studied \cite{BGM,Curtis,May}. In particular,  fibre bundles with fixed fibre and base space are classified, up to fibrewise equivalence. We   recall here the notion of twisting function and twisted cartesian products, for later use. 

\begin{defi} 

A \emph{fibre bundle} of simplicial sets is a simplicial map $f\colon E\Right0{} B$ satisfying the following conditions:
\begin{enumerate}[\rm (a)]

\item $f$ is surjective; and

\item  there is a simplicial set $F$ such that, for any simplex $\7b\colon \Delta[n] \Right0{}B$, there is a isomorphism $\phi_{\7b}\colon F\times \Delta[n] \Right0{} E_{\7b}$ making the following diagram 
\begin{equation}\label{bundlecondition}
 \xymatrix{
F\times \Delta[n] \ar[r]^-{\phi_{\7b}}_-\cong  \ar[dr] & E_{\7b}  \ar[d] \ar[r]^{\tilde{\7b}} &  E\ar[d]^f \\
 &   \Delta[n] \ar[r]^{\7b} &  B
}
\end{equation}
where $E_{\7b}\Right0{}\Delta[n]$ is the pullback of $E\Right0{}B$ along ${\7b}$.

\end{enumerate}
$F$ is called the \emph{fibre} of the bundle, $B$ is the \emph{base space} and $E$ the \emph{total space}.

\end{defi}

\begin{defi}

Let $\locm$ and $\loch$ be partial groups. An \emph{extension} of $\loch$ by $\locm$ is a fibre bundle $\tau \colon\loce \to \loch$ with fibre $\locm$.

\end{defi}

\begin{defi}
Two extensions $\locm\Right0{}\loce_i\Right0{}\loch$, $i=1,2$, are 
equivalent if  there is an isomorphism $f\colon \loce_1\Right0{}\loce_2$ 
such that the diagram 
$$
\xymatrix{\locm \ar[r] \ar@{=}[d] &  \loce_1 \ar[r]\ar[d]^\varphi &  \loch \ar@{=}[d] \\
            \locm \ar[r]  &  \loce_2 \ar[r]&  \loch
}
$$
is commutative.

\end{defi}

This definition is consistent with the concept of fibre homotopy equivalence of fibre bundles (strong homotopy equivalence 
in \cite{BGM}). Two fibre bundles $\loce_1 \Right0{\tau_1}  \loch$
 with the same fibre $\locm$ are fibrewise homotopy equivalent if there 
 are maps   $f\colon \loce_1\Right0{}\loce_2$ and  
 $g\colon \loce_2\Right0{}\loce_1$ over $\loch$, (i.e.,
   $\tau_1 = \tau_2\circ f$ and  $\tau_1\circ g = \tau_2$) 
   such that $f\circ g$ and $g\circ f$ 
   are homotopic to the identity over $\loch$. 
   According to Lemma~\ref{he=iso}, both $f$  and $g$ are isomorphisms.

Given an extension of partial groups $\locm\Right0{}\loce\Right0{\tau}\loch$, we can easily prove 
that the total space $\loce$ is a partial monoid. Indeed, It is clear that $\loce$ has a unique vertex since so happens with $\locm$ and $\loch$. 
Next, let $\7x$ and $\7y$ be two $n$-simplices of $\loce$ with same spine, 
$\7e^n(\7x) = \7e^n(\7y)$. Then, the same is true for $\tau(\7x)$ and $\tau(\7y)$, but 
$\loch$ is a partial group, so they are equal. Write $\7b=\tau(\7x)=\tau(\7y)$ for simplicity.
Following the notation of diagram~\eqref{bundlecondition}, we have 
$\7x\,,\,\7y\in \loce_{\7b}$. Now $\loce_{\7b}\cong \locm\times\Delta[n]$ has injective spine operators, hence $\7x=\7y$. 

Proving that $\loce$ has also an inversion, and therefore it is a partial group, might not be 
so apparent at first glance. However, it follows at once from the structure Theorem for extensions 
of partial groups. This constitutes the main result of this section, stated as Theorem \ref{twistedn-simpl} below.

\begin{defi}\label{rmktwistp} 

Let $\locm$ and $\loch$ be partial groups. An \emph{$\loch$-twisting pair for $\locm$} is a pair of functions $(t,\eta)$,
$$
t \colon \loch_1 \Right3{} \Aut(\locm) \qquad \mbox{and} \qquad \eta \colon \loch_2 \Right3{} N(\locm),
$$
satisfying the following conditions:
\begin{enumerate}[\rm (a)]
\item $\eta(g,h)$ determines a homotopy $t(g)\circ t(h) \Left2{\eta(g,h)} t(gh)$, for each pair $[g|h] \in \loch_2$;

\item $t(1)=\Id$ and $\eta(g,1)=1=\eta(1,g)$, for all $g\in \loch_1$; and

\item (cocycle condition) for all $[g|h|k] \in \loch_3$,
$$
t(g)(\eta(h,k)) \cdot \eta(g, hk) = \eta(g,h) \cdot \eta(gh,k).
$$
\end{enumerate}
\end{defi}

In particular, an $\loch$-twisting pair for $\locm$ determines an \emph{outer action} of $\loch$ on $\locm$, that is, a homomorphism (of partial groups) $\loch \to B\Out(\locm)$.

\begin{thm}\label{twistedn-simpl}

Let $\locm$ and $\loch$ be partial groups. Then, the following holds.
\begin{enumerate}[\rm (a)]

\item An $\loch$-twisting pair $\phi=(t,\eta)$  for $\locm$ defines an extension $\locm \times_{\phi} \loch$ of $\loch$ by $\locm$, which is a partial group with $n$-simplices $[(x_1, g_1)|\ldots |(x_n, g_n)]$ satisfying the conditions
\begin{enumerate}[\rm (i)]

\item $[g_1|\ldots|g_n] \in \loch_n$; and

\item $[x_1| t(g_1)(x_2)| (t(g_1) \circ t(g_2))(x_3)| \ldots | (t(g_1) \circ \ldots\circ  t(g_{n-1}))(x_n)] \in \locm_n$.

\end{enumerate}
The product and inverse formulae for $\locm \times_{\phi} \loch$ are, respectively,
\begin{equation}\label{cond3}
\Pi[(x,g)|(z,h)] = [(x\cdot t(g)(z)\cdot \eta(g,h), g \cdot h)]
\end{equation}
\begin{equation}\label{cond4}
[(x, g)]^{-1} = [(\eta(g^{-1},g)^{-1}\cdot t(g^{-1})(x^{-1}),g^{-1})]\,,
\end{equation}
in terms of products and inverses in $\locm$ and in $\loch$. The projection $\tau \colon \locm\times_\phi \loch \Right1{} \loch$ is defined on 
$n$-simplices as $\tau[(x_1,g_1)|\ldots|(x_n,g_n)] = [g_1|\ldots|g_n]$.

\item Given an extension $\locm\Right1{}\loce\Right1{\rho} \loch$, there is an $\loch$-twisting pair $\phi=(t,\eta)$  for $\locm$, and an equivalence of extensions
$$
\xymatrix{  \locm\times_\phi BG  \ar[d]_{\tau} \ar[r]^-\cong & \loce \ar[d]^\rho \\
\loch \ar@{=}[r] & \loch
}
$$

\end{enumerate}
In particular, the total space of an extension of partial groups is again a partial group.

\end{thm}

The rest of this section is devoted to the proof of this result. This is based on the theory of fibre bundles 
of simplicial sets developed in \cite{BGM}. Roughly speaking, a fibre bundle is equivalent to a twisted 
cartesian product, and this is determined by a twisting function. In Lemma~\ref{tpair}, 
we show that in the case of partial groups there is a bijective correspondence between 
twisting functions and twisting pairs as defined in Definition~\ref{rmktwistp}. 

In view of Theorem~\ref{twistedn-simpl}, the classification of extensions can be carried out 
following the lines of the classical case of finite groups. This is made precise in Theorem~\ref{classext}.

\medskip




Recall from \cite{BGM} that given simplicial sets $F$ and $B$, 
a \emph{twisting function} $\phi\colon B \Right0{} \aut(F)$ is a collection of maps
$$\phi_n\colon B_n \Right2{} \aut(F)_{n-1}\,, \qquad n\geq 1$$
satisfying the following conditions for all $\7b \in B_n$:
\begin{equation}\label{tau}
\begin{aligned}
\phi_{n-1}(d_i\7b) &= d_{i-1}\phi_n(\7b)\,, &  & \qquad 2 \leq i \leq n,\\
\phi_{n-1}(d_1\7b) & = d_{0}\phi_n(\7b)\cdot \phi_{n-1}(d_0\7b),\\
\phi_{n+1}(s_i\7b) &= s_{i-1}\phi_n(\7b)\,,&  & \qquad i \geq 1, \\
\phi_{n+1}(s_0\7b) & = \7 1 \in \Gamma_n. 
\end{aligned}
\end{equation}

A twisting function $\phi\colon B \Right0{} \aut(F)$ determines a fibre bundle over $B$ with  
fibre $F$, where the total space is the twisted cartesian product $F\times_\phi B$, namely, the simplicial set 
 with $n$-simplices 
$(F\times_\phi B)_n = F_n\times B_n$ and with face and degeneracy maps defined as
\begin{align*}
 d_0(\7x,\7b) & =   ( \phi(\7b)^{-1} \cdot d_0\7x, d_0\7b) \,,\\
 d_i(\7x,\7b) & =   (d_i\7x,d_i\7b)      & & i\geq1 \,, \\
  s_i(\7x,\7b) & =   (s_i\7x,s_i\7b)      & & i\geq0\,,
\end{align*}
for all simplices $\7x\in F_n$, $\7b\in B_n$, and $n\geq0$. The natural projection $F\times_\phi B \Right0{} B$ is a fibre bundle with fibre $F$. The following well known result appears in \cite[Proposition IV.3.2]{BGM}

\begin{thm}\label{BGM432}
Given a fibre bundle $\tau \colon E\Right0{} B$ with fibre $F$, there is a twisting function 
$\phi\colon B \Right0{} \aut(F)$ and a fibrewise equivalence $E\cong F\times_\phi B$. 
\end{thm}




In the next Lemma we carry out the precise relationship between the twisting functions for partial groups
and twisting pairs as defined in \ref{rmktwistp}.

%
%
%
%
%
%
%

\begin{lmm} \label{tpair}

Let $\locm$ and $\loch$ be partial groups. An $\loch$-twisting pair $(t,\eta)$ for $\locm$ determines a unique twisting function $\phi\colon \loch \Right0{}  \nv{\calaut(\locm)}$ with 
\begin{equation}\label{settf0}
\begin{split}
  \phi_1[g] &= t(g)\in \Aut(\locm)\,,\quad\text{and} \\ 
\phi_2[g|h] &=
          \bigl( t(g)\Left2{\eta(g,h)}t(gh)t(h)^{-1}  \bigr ) \in\Mor(\calaut(\locm))
\end{split}
\end{equation}
Moreover, every twisting function $\phi\colon \loch\Right0{}  \nv{\calaut(\locm)}$ is of this form.

\end{lmm}

\begin{proof}

Let $(t, \eta)$ be an $\loch$-twisting pair for $\locm$. For each $\7h = [h_1|\ldots|h_m] \in \loch_m$ and for each $1 \leq i \leq j \leq n$, set $h_{i,j} = h_i \cdot \ldots \cdot h_j$. Also, set
\begin{equation}\label{settf}
\begin{split}
\alpha_0 & = t(h_1)\\
\alpha_j & = t(h_{1, j+1}) \circ t(h_{2, j+1})^{-1} \qquad j = 1, \ldots, m-1\\
\eta_0 & = \eta(h_1, h_2) \\
\eta_j & = \eta(h_1, h_{2,j})^{-1} \cdot \eta(h_1, h_{2, j+1}) \qquad j = 1, \ldots, m-2\\
\end{split}
\end{equation}
and define
\begin{equation}\label{settf2}
 \phi_m[h_1|\ldots|h_m] = (\alpha_0 \Left2{\eta_0} \ldots \Left2{\eta_{m-2}} 
 \alpha_{m-1}) \in |\calaut(\locl)|_{m-1}
\end{equation}
for $m\geq1$, thus extending (\ref{settf0}).

For $k = 2, \ldots, m$, the element $\eta(h_1, h_{2,k})$ defines a homotopy $t(h_1) \Left2{\eta(h_1, h_{2,k})} t(h_{1, k}) \circ t(h_{2,k})^{-1}$ by Definition \ref{rmktwistp} (a). Thus, for each $j = 0, \ldots, m-2$ there are homotopies
$$
\alpha_j \Left3{\eta_j} \alpha_{j+1}.
$$
This shows that $\phi =\{\phi_n\}_{n\geq1}$ is a well defined function.

We claim that it is indeed a twisting function. The conditions listed in \eqref{tau} all follow by inspection except for the second one:
\begin{equation}\label{condii}
 \phi_{n-1}(d_1\7g) = d_0\phi_n(\7g)\cdot \phi_{n-1}(d_0\7g)
\end{equation}
 for an arbitrary $n$-simplex  $\7g=[g_1|\dots|g_n]$, $n\geq2$,
depends on the cocycle condition, (c) of Lemma~\ref{rmktwistp}. According to the formulae \eqref{settf}, we have
\begin{enumerate}[(A)]

\item $d_0\phi_n(\7g) = (\alpha_1 \Left2{\eta_1} \ldots \Left2{\eta_{n-2}} \alpha_{n-1})$, where
$$
\begin{aligned}
\alpha_j & = t(g_{1,j+1}) \circ t(g_{2,j+1})^{-1} \qquad j = 1, \ldots, n-1\\
\eta_j & = \eta(g_1, g_{2,j+1})^{-1} \cdot \eta(g_1, g_{2, j+2}) \qquad j = 1, \ldots, n-2\\ 
\end{aligned}
$$

\item $\phi_{n-1}(d_0\7g) = (\beta_1 \Left2{\zeta_1} \ldots \Left2{\zeta_{n-2}} \beta_{n-1})$, where
$$
\begin{aligned}
\beta_1 & = t(g_2) \\
\beta_j & = t(g_{2, j+1}) \circ t(g_{3, j+1})^{-1} \qquad j = 2, \ldots, n-1\\
\zeta_1 & = \eta(g_2,g_3)\\
\zeta_j & = \eta(g_2, g_{3,j+1})^{-1} \cdot \eta(g_2, g_{3,j+2}) \qquad j = 2, \ldots, n-2
\end{aligned}
$$

\item $ \phi_{n-1}(d_1\7g)  = (\gamma_1 \Left2{\mu_1} \ldots \Left2{\mu_{n-2}} \gamma_{n-1})$, where
$$
\begin{aligned}
\gamma_1 & = t(g_{1,2})\\
\gamma_j & = t(g_{1,j+1}) \circ t(g_{3,j+1})^{-1} \qquad j = 2, \ldots, n-1\\
\mu_1 & = \eta(g_{1,2}, g_3)\\
\mu_j & = \eta(g_{1,2}, g_{3,j+1})^{-1} \cdot \eta(g_{1,2}, g_{3,j+2}) \qquad j = 2, \ldots, n-2
\end{aligned}
$$

\end{enumerate}
Moreover, the product in the right hand side of (\ref{condii}) is the simplicial group operation induced by the product $\otimes$ defined on $\calaut(\locm)$  by equation~(\ref{tensor}):
$$
d_0\phi_n(\7g)\cdot \phi_{n-1}(d_0\7g) = (\alpha_1 \circ \beta_1 \Left6{\eta_1 \alpha_2(\zeta_1)} \ldots \Left9{\eta_{n-2} \alpha_{n-1}(\zeta_{n-2})} \alpha_{n-1} \circ \beta_{n-1}).
$$
It is immediate to check that $\alpha_j \circ \beta_j = \gamma_j$ for $j = 1, \ldots, n-1$.

Thus, in order to check \eqref{condii} it remains to show that
$$
\eta_j \cdot \alpha_{j+1}(\zeta_j) = \mu_j
$$
for $j = 1, \ldots, n-2$. Let $j = 1$ first. Then, we have
\begin{align*}
\eta_j \cdot \alpha_{j+1}(\zeta_j) & = \eta(g_1,g_2)^{-1} \cdot \eta(g_1, g_{2,3}) \cdot (t(g_{1,3}) \circ t(g_{2,3})^{-1})(\eta(g_2,g_3)) = \\
 & = \eta(g_1,g_2)^{-1} \cdot \eta(g_1, g_{2,3}) \cdot (\cj{{\eta(g_1, g_{2,3})^{-1}}} \circ t(g_1))(\eta(g_2,g_3)) = \\
 & = \eta(g_1,g_2)^{-1} \cdot \eta(g_1, g_{2,3}) \cdot \eta(g_1, g_{2,3})^{-1} \cdot t(g_1)(\eta(g_2,g_3)) \cdot \eta(g_1, g_{2,3}) = \\
 & = \eta(g_{1,2},g_3) = \mu_1
\end{align*}
The equality between lines one and two is justified as follows: by Definition \ref{rmktwistp} (a), the element $\eta(g,h)$ defines a homotopy $t(g) \circ t(h) \Left1{\eta(g,h)} t(gh)$; and the equality now follows by Lemma \ref{ho.eq} (b), together with Definition \ref{definorm}. The equality between lines three and four follows by the cocycle condition in Definition \ref{rmktwistp} (c).

A similar argument now shows that, for $j = 2, \ldots, n-2$, we have:
\begin{align*}
\eta_j \cdot & \alpha_{j+1}(\zeta_j) = \\
 & = \eta(g_1, g_{2,j+1})^{-1} \cdot \eta(g_1, g_{2, j+2}) \cdot (t(g_{1,j+2}) \circ t(g_{2,j+2})^{-1})(\eta(g_2, g_{3,j+1})^{-1} \cdot \eta(g_2, g_{3,j+2})) = \\
 & = \eta(g_1, g_{2,j+1})^{-1} \cdot \eta(g_1, g_{2, j+2}) \cdot (\cj{{\eta(g_1, g_2,j+2)^{-1}}} \circ t(g_1))(\eta(g_2, g_{3,j+1})^{-1} \cdot \eta(g_2, g_{3,j+2})) = \\
 & = \eta(g_1, g_{2,j+1})^{-1} \cdot t(g_1)(\eta(g_2, g_{3,j+1})^{-1}) \cdot t(g_1)(\eta(g_2, g_{3,j+2})) \cdot \eta(g_1, g_{2, j+2}) = \\
 & = \eta(g_{1,2}, g_{3,j+1})^{-1} \cdot \eta(g_1,g_2)^{-1} \cdot \eta(g_1, g_2) \cdot \eta(g_{1,2}, g_{3,j+2}) = \mu_j
\end{align*}
The equality between lines four and five follows from two applications of the cocycle condition in Definition \ref{rmktwistp} (c) (one for the first two terms in line three, and another one for the last two terms). This proves equality \eqref{condii}, and thus $\phi$ as defined in \eqref{settf2} is a twisting function.

To finish the proof we have to show that every twisting function is defined in this way. Fix a twisting function $\phi \colon \loch \to |\calaut(\locm)|$ and define functions
$$
t \colon \loch_1 \Right2{} \Aut(\locm)  \qquad \mbox{and} \qquad \eta \colon \loch_2 \Right2{} N(\locm)
$$
as follows:
\begin{itemize}
\item $t(g) = \phi_1[g]\in  |\calaut(\locm)|_0 = \Aut(\locm)$,  for each $[g] \in \loch_1$; and
\item $\eta(g,h)\in N(\locm)$ is the label of the morphism
$$
\phi_2[g|h] = 
\bigl(t(g) \Left2{\eta(g,h)} t(gh) \circ t(h)^{-1}\bigr) \in  |\calaut(\locm)|_0 =  \Mor_{\calaut(\locm)}
$$ 
for each $[g|h] \in \loch_2$.

\end{itemize}
Notice that, according to the definition of $\nv{\calaut(\locm)}$ and the first two conditions in \eqref{tau}, $\phi_2[g|h]$ is really a homotopy from $d_0(\phi_2   ([g|h]))  =  \phi_1(d_2[g|h] )    \phi_1(d_0[g|h] )^{-1} = t(gh)t(h)^{-1}$ to $d_1\phi_2([g|h]) = \phi_1(d_2[g|h]) = t(g)$. We have to check that these functions satisfy conditions (a), (b) and (c) of Definition \ref{rmktwistp}, and that $\phi$ is determined by this twisting pair and the formulae (\ref{settf2}).

\textbf{Step 1.} Condition (a) of Definition \ref{rmktwistp}, follows by definition of $\eta$ and 
\ref{ho.eq}(b). Similarly, by the last two conditions in \eqref{tau}, we have
$\phi_2[g|1] = \phi_2(s_1[g]) =  s_0(\phi_1[g]) = s_0(t(g)) = (t(g) \Left2{1} t(g))$ and 
$\phi_2[1|g] = \phi_2(s_0[g]) =  (\Id \Left2{1} \Id)$,
thus, $\eta(g,1) = 1 = \eta(1,g)$. This is condition (b) in \ref{rmktwistp}.

To check condition (c) in \ref{rmktwistp}, fix $[g|h|k] \in \loch_3$ and set $\phi_3[g|h|k] = (\alpha_0 \Left2{\eta_1} \alpha_1 \Left2{\eta_2} \alpha_2)$. Applying property the first condition in \eqref{tau} with $i = 3, 2$ respectively, we get
\begin{align*}
\bigl(t(g) \Left3{\eta(g,h)} t(gh) \circ t(h)^{-1}\bigr) & = (\alpha_0 \Left2{\eta_1} \alpha_1) \\
\bigl(t(g) \Left3{\eta(g, hk)} t(ghk) \circ t(hk)^{-1}\bigr) & = (\alpha_0 \Left2{\eta_1 \eta_2} \alpha_2)
\end{align*}
Thus, $\phi_3[g|h|k] = \bigl(t(g) \Left3{\eta(g,h)} t(gh) \circ t(h)^{-1} \Left9{\eta(g,h)^{-1} \eta(g,hk)} t(ghk) \circ t(hk)^{-1}\bigr)$.

On the other hand, applying the second condition in \eqref{tau}, 
and the product structure of $|\calaut(\locm)|_1$, we get the following:
\begin{align*}
(\alpha_1 \Left2{\eta_2} \alpha_2) & = d_0(\phi_3[g|h|k]) = \phi_2[gh|k] \otimes \phi_2[h|k]^{-1} = \\
 & = (t(gh) \Left3{\eta(gh,k)} t(ghk) \circ t(k)^{-1}) \otimes (t(h) \Left3{\eta(h,k)} t(hk) \circ t(k)^{-1})^{-1} = \\
 & = (t(gh) \circ t(h)^{-1} \,\, \Left{9}{(t(gh) \circ t(h)^{-1})(\eta(h,k)^{-1}) \eta(gh,k)} \,\, t(ghk) \circ t(hk)^{-1}).
\end{align*}
Note that $\eta(g,h)$ determines a homotopy $(t(g) \Left2{\eta(g,h)} t(gh) \circ t(h)^{-1})$, and in particular this means that
$$
(t(gh) \circ t(h)^{-1})[y] = \eta(g,h)^{-1} \cdot t(g)[y] \cdot \eta(g,h)
$$
for all $[y] \in \locm_1$. Combining this with the above description of $\phi_3[g|h|k]$, we deduce that
\begin{align*}
\eta(g,h)^{-1} \cdot \eta(g,hk) & = \eta_2 = (t(gh) \circ t(h)^{-1})(\eta(h,k)^{-1}) \eta(gh,k) = \\
 & = \eta(g,h)^{-1} \cdot t(g)(\eta(h,k)^{-1}) \cdot \eta(g,h) \cdot \eta(gh,k).
\end{align*}
Property (c) of twisting pairs \ref{rmktwistp} follows immediately. 

To finish the proof we have to show that $\phi_n[g_1|\ldots|g_n]$ is given by the formulae (\ref{settf})
above, and this is actually easily checked inductively by applying the first condition in \eqref{tau} with $i = n-1, n$ to 
$\phi_n[g_1|\ldots|g_n]$. Recall that the simplicial group $|\calaut(\locm)|$ has underlying simplicial set the nerve of the category
$\calaut(\locm)$. In particular the spine operator 
\begin{multline*}
 \7e^n  \colon \bigl(\alpha_0\Left1{\eta_1}\Left1{}\dots \Left1{\eta_n}\alpha_n\bigr) 
      \in |\calaut(\locm)|_n  \\
\longmapsto
\bigl( (\alpha_0\Left1{\eta_1}\alpha_1), \dots , (\alpha_{n-1}\Left1{\eta_n}\alpha_n)\bigr) \in 
 |\calaut(\locm)|_1 \times \ldots \times |\calaut(\locm)|_1  \,. 
\end{multline*}
is injective for all $n\geq 1$. It   be convinient to observe that for $n\geq2$, 
$ \7e^n $ factors
as the composition
$$
 |\calaut(\locm)|_n     \Right4{(\7e^n_1, d_0)}   |\calaut(\locm)|_1   \times   |\calaut(\locm)|_{n-1} 
  \Right4{1\times\7e^{n-1}}   |\calaut(\locm)|_1 \times \ldots \times |\calaut(\locm)|_1 
 $$
where both maps are also injective.

Assume now, that we have two twisted functions $\phi_n, \psi_n\colon \loch\Right1{} |\calaut(\locm)|$, 
such that $\phi_i=\psi_i$ for $i=1,2$.
Assume iductively that $\phi_i=\psi_i$ for $i\leq n$, for $n\geq2$. Then choose an arbitrary 
$(n+1)$-simplex $\7g$ of $\loch$, and compare 
$\phi_{n+1}(\7g)$ and $\psi_{n+1}(\7g)$.

Notice that $\7e_1^{n}=d_2\circ d_3\circ\dots\circ d_n$, thus 
\begin{align*}
 e_1^{n}(\phi_{n+1}(\7g)) & = d_2\circ d_3\circ\dots\circ d_n(\phi_{n+1}(\7g)) =
 d_2\circ d_3\circ\dots\circ d_{n-1}(\phi_{n}(d_{n+1}\7g)) = \dots \\
& = \phi_2((d_3\circ\dots\circ d_n)\7g)\,\ \text{and also }\\
d_0(\phi_{n+1}(\7g)) &= \phi_{n}(d_1(\7g)) \cdot \phi_{n}(d_0\7g)^{-1}
\end{align*}
A similar computation holds for $\psi_{n+1}$ and then, by the induction hypothesis, 
$\7e_1^{n}(\phi_{n+1}(\7g)) =  \7e_1^{n}(\psi_{n+1}(\7g))$ and 
$d_0(\phi_{n+1}(\7g)) = d_0(\psi_{n+1}(\7g)) $. Since $(\7e_n^1,d_0)$, defined on 
$ |\calaut(\locm)|_n $ is injective, it follows that $\phi_{n+1}(\7g)=\psi_{n+1}(\7g)$.
This finishes the proof.
\end{proof}
\medskip


\begin{proof}[Proof of Theorem~\ref{twistedn-simpl}]

Set $\loce=\locm\times_\phi \loch$ for short, and note that $\loce$ contains a single vertex, 
corresponding to the pair $(v_{\locm}, v_{\loch}) \in \locm_0 \times \loch_0$. The $1$-simplices 
of $\loce$ are pairs $(x,g)\in \locm_1\times \loch_1$. The $2$-simplices are pairs 
$\sigma = ([x|y], [g|h])\in \locm_2\times \loch_2$. The edges operator on $\sigma$ 
has the following effect
$$
\7e^2(\sigma) = (d_2(\sigma) , d_0(\sigma))  = \bigl((x,g),  (\phi_2([g,h])^{-1}\cdot y, h) \bigr),
$$
and we can express the simplex $\sigma$ of $\loce$ as
$$
\sigma = [(x,g)| (\phi_2([g,h])^{-1}\cdot y, h)].
$$
By writing $z= \phi_2([g,h])^{-1}\cdot y$, so that $y= \phi_2([g,h])\cdot z = t(g)(z)\cdot \eta(g,h)$, we  can identify the 2-simplices of $\loce$ with all pairs
$$
[(x,g)|(z,h)]
$$
for which $[g|h] \in \loch_2$ and $[x| t(g)(z)\cdot \eta(g,h)]\in \locm_2$. The general 
expression for the simplices of $\locm \times_{\phi} \loch$ follows by induction. In particular, 
it is easy to see that the edge operator is injective on $\loce$ for all $n \geq 2$, and thus 
$\loce = \locm \times_{\phi} \loch$ is a partial group.

The product in $\loce$ is easily deduced now:
$$
(x,g)\cdot(z,h) = d_1([(x,g)|(z,h)]) = (x\cdot t(g)(z)\cdot \eta(g,h), gh)\,,
$$
which is formula \eqref{cond3}. Regarding the inversion on $\loce$, for each $[(x,g)]\in \loce_1$ we can define
$$
[(x, g)]^{-1} = [(\eta(g^{-1},g)^{-1}\cdot t(g^{-1})(x^{-1}),g^{-1})].
$$
Note that $[\eta(g^{-1},g)^{-1}\cdot t(g^{-1})(x^{-1}) | t(g^{-1})(x)\cdot \eta(g^{-1},g)] 
\in \locm_2$, since $\locm$ has an inversion and $\eta(g^{-1},g)\in N(\locm)$. 
This way it follows that
$$
[(\eta(g^{-1},g)^{-1}\cdot t(g^{-1})(x^{-1}),g^{-1})|(x, g) ]\in \loce_2,
$$
and the product formula yields $\Pi([(\eta(g^{-1},g)^{-1}\cdot t(g^{-1})(x^{-1}),g^{-1})|(x, g) ])=1$.
\end{proof}

Given a partial group $\loce$ and a partial subgroup $\locm \leq \loce$, let $\locn_{\loce}(\locm)$ be the set of simplices $[\eta_1|\ldots|\eta_n] \in \loce$ such that $\eta_i$ defines a homotopy of $\incl \colon \locm \to \loce$, for each $i = 1, \ldots, n$. There is no reason a priori to assume that $\locn_{\loce}(\locm)$ is a partial subgroup of $\loce$.

\begin{lmm}\label{NEL}
Let $\locm \Right1{\iota} \loce \Right1{\tau} \loch$ be an extension of partial groups. Then, the following holds.
\begin{enumerate}[\rm (a)]

\item The above extension restricts to an extension 
$$
      BN(\locm) \Right4{\iota|_{BN(\locm) }} \locn_{\loce}(\locm) 
                     \Right4{\tau|_{\locn_{\loce}(\locm)}} \loch\,,
$$
where $\locn_{\loce}(\locm)$ is a partial subgroup of $\loce$. 

\item 
Each element $z\in \locn_{\loce}(\locm)_1$ determines an automorphism $\7c_z\in\Aut(\locm)$
~and a homotopy:
 $$\iota\circ\7c_z\Left2{z} \iota$$

\item Restriction from $\loch$ to $N(\loch)$ induces a group extension 
$N(\locm) \Right1{} N(\loce) \Right1{} N(\loch)$.

\end{enumerate}
%
Henceforth, we shall refer 
to $\locn_{\loce}(\locm)$ as the \emph{normalizer partial subgroup} of $\locm$ in $\loce$.
\end{lmm}

\begin{proof}

According to Theorem~\ref{twistedn-simpl}, we may identify $\loce = \locm \times_{\phi} \loch$, with $\loch$-twisting pair $\phi = (t, \eta)$ for $\locm$. This way, we may use the usual notation and conventions for twisted cartesian products.

Notice that $N(\locm)$ is a characteristic subgroup of $\locm$ by Proposition \ref{Nloclgroup} - that is, every automorphism of $\locm$ preserves $N(\locm)$. Moreover, the function $\eta$ in the twisting pair $\phi$ takes values in $N(\locm)$. Hence, part (a) of the statement follows immediately.

To check part (b) of the statement, let $(x,g) \in \locn_{\loce}(\locm)_1$, and notice that $x \in N(\locm)$ by definition. As $x$ acts by conjugation on $\locm$ by Proposition \ref{Nloclgroup}, we may assume without loss of generality that $x = 1$: the conjugation action of $(x,g)$ is equal to the composition of the conjugation actions of $(x,1)$ and $(1,g)$.

Given $\7w = [y_1|\ldots|y_n] \in \locm$, we have to show that
$$
\5{\7w} = [(1,g)|(y_1,1)|(1,g)^{-1}|(1,g)|(y_2,1)|(1,g)^{-1}|\ldots|(1,g)|(y_n,1)|(1,g)^{-1}] \in \loce.
$$
This is a tedious computation combining conditions (i) and (ii) in Theorem \ref{twistedn-simpl} (a), as well as the inversion formula \eqref{cond4}, and the different properties listed in Lemmas \ref{ho.eq} and \ref{N(M)conj}. We leave the details to the reader.

To prove part (c) of the statement, we first show that $(1, g), (x, 1) \in N(\loce)$ for all $g \in N(\loch)$ and each $x \in N(\locm)$.  The proof of these two steps is based on 
Lemma~\ref{N(M)conj}.  It is also convenient to emphasize the following technical fact. 
Given $g \in \loch$, recall from \eqref{cond4} that $(1, g)^{-1} = (\eta(g^{-1}, g)^{-1}, g^{-1})$. Set for simplicity $\eta = \eta(g^{-1},g)^{-1} \in N(\locm)$ (cf.~\eqref{cond4}).  
With this notation, $\eta$ determines a homotopy 
$\Id_\locm   \Left2{\eta} t(g^{-1}) \circ t(g)$, and by Lemma \ref{ho.eq} (a) and (c), for any 
automorphism $f$ of $\locm$,  $f(\eta)$ defines a homotopy
$$
f \Left4{f(\eta)} f \circ t(g^{-1}) \circ t(g).
$$
Hence, given a simplex $\7x=[x_1|\dots|x_{k-1}|f(x_k)|\dots |f(x_n)] \in \locm$, the above homotopy 
shows the existence of another simplex 
\begin{equation}\label{tech2}
 [x_1|\dots|x_{k-1}| f(\eta) | f\circ t(g^{-1})\circ t(g) (x_k)|\dots | f\circ t(g^{-1})\circ t(g) (x_n)]
\end{equation}
(cf.~Lemma~\ref{homot=conj}(b), applied to $[f^{-1}(x_1)|\dots |f^{-1}(x_{k-1})|x_k|\dots\ x_n]$.)
\medskip

\textbf{Step 1. } For any $g\in N(\loch)$,  $(1, g) \in N(\loce)$. By Lemma~\ref{N(M)conj}\eqref{N(M)conj1} it is 
enough to show
that for an arbitrary simplex
$$
  \7w = [(y_1, h_1)|\ldots|(y_n, h_n)] \in \loce
$$ 
there are also simplices 
$$
\5{\7w}_k= [(y_1, h_1)|\dots (y_{k-1},h_{k-1})|(1, g)^{-1}|(1, g)|(y_k, h_k)| 
             \ldots |(y_n, h_n)] \in \loce\,,\quad k=1,\dots,n+1\,.
$$
Fix $k$, $1\leq k\leq n+1$. To prove the existence of the simplex $\5{\7w}_k$ we have to check conditions (i) and (ii) in Theorem \ref{twistedn-simpl} (a). Notice that condition (i) is immediate as $g \in N(\loch)$. Thus, $\5{\7w}_k \in \loce$ if and only if condition (ii) in Theorem \ref{twistedn-simpl} (a) (for $\5{\7w}_k$) is satisfied. Notice that we have $[y_1|t(h_1)(y_2)| \ldots |(t(h_1) \circ \ldots t(h_{n-1}))(y_n)] \in \locm_n$ by assumption, as $\7w \in \loce$.  Then, 
equation~\eqref{tech2} with $f= t(h_1)\circ \dots \circ t(h_{k-1})$ 
shows the existence of the simplex 
\begin{align*}
\7v_{k}= [y_1&|t(h_1)(y_2)|\ldots\; \ldots  |t(h_1) \circ \ldots \circ t(h_{k-1})(y_{k-1})  
                        |  t(h_1) \circ \ldots  \circ t(h_{k-1}) (\eta)   \\  
                                                 & | t(h_1) \circ \ldots \circ  t(h_{k-1}) \circ t(g^{-1})\circ t(g)(y_k)|
                                t(h_1) \circ \ldots t(h_{k-1}) \circ t(g^{-1})\circ t(g)\circ t(h_k)(y_{k+1})|\dots \\
           &  \qquad  \dots     |  t(h_1) \circ \ldots \circ  t(h_{k-1}) \circ t(g^{-1})\circ t(g)
           \circ t(h_k)\circ \dots \circ t(h_{n-1} )(y_n)]\in \locm\,.
\end{align*}
On the other hand, by condition (ii) in Theorem \ref{twistedn-simpl} (a) we know that $\5{\7w}_k \in \loce$ if and only if
\begin{align*}
\7u_{k}= [y_1&|t(h_1)(y_2)|\ldots\; \ldots  |t(h_1) \circ \ldots \circ t(h_{k-1})(y_{k-1})  \\ 
                      &|    t(h_1) \circ \ldots  \circ t(h_{k-1}) (\eta)|  
                        t(h_1) \circ \ldots  \circ t(h_{k-1}) \circ t(g^{-1})(1)\\  
                                                 & | t(h_1) \circ \ldots \circ  t(h_{k-1}) \circ t(g^{-1})\circ t(g)(y_k)|
                                t(h_1) \circ \ldots t(h_{k-1}) \circ t(g^{-1})\circ t(g)\circ t(h_k)(y_{k+1})|\dots \\
           &  \qquad  \dots     |  t(h_1) \circ \ldots \circ  t(h_{k-1}) \circ t(g^{-1})\circ t(g)
           \circ t(h_k)\circ \dots \circ t(h_{n-1} )(y_n)]\in\locm\,.
\end{align*}
Notice that $ t(h_1) \circ \ldots  \circ t(h_{k-1}) \circ t(g^{-1})(1)=1$ thus, 
$ \7v_{k}$ is a degeneracy of $\7u_{k}$, more precisely  $\7u_{k} = s_{k+1} (\7v_{k}) $. 
Hence, we have shown the existence of $\5{\7w}_k$.
\medskip

\textbf{Step 2. }  For any $x\in N(\locm)$, $(x,1)\in N(\loce)$. Again by Lemma~\ref{N(M)conj}\eqref{N(M)conj1},
we only need to show that 
for each $\7w =[(y_1,h_1)|\dots |(y_n, h_n)]\in\loce$ and each  $k$, $1\leq k\leq n+1$, there is a simplex
$$
\5{\7w}_k= [(y_1,h_1)|\dots      | (y_{k-1},h_{k-1})|(x,1)^{-1}| (x,1)  |(y_k, h_k)|\dots    |(y_n, h_n)]  \in \loce \,.
$$
Again, that amounts to showing that conditions (i) and (ii) in Theorem \ref{twistedn-simpl} (a) are satisfied for $\5{\7w}_k$. Condition (i) follows immediately, as the projection of $(x, 1)$ onto $\loch$ is the identity element.

As $\7w\in \loce$, we have a simplex
$[y_1|t(h_1)(y_2)| \ldots |(t(h_1) \circ \ldots t(h_{n-1}))(y_n)] \in \locm$ by condition (ii) in Theorem \ref{twistedn-simpl} (a). Moreover, $t(h_1)\circ\dots\circ t(h_{k-1})(x) \in N(\locm)$, and hence we have a simplex
\begin{align*}
 [y_1|t(h_1)(y_2)& | \ldots\;\dots | t(h_1) \circ \ldots \circ t(h_{k-2})(y_{k-1})  \\
            &  | t(h_1) \circ \ldots \circ t(h_{k-1})(x^{-1}) |        t(h_1) \circ \ldots \circ t(h_{k-1})(x)     \\
             &   | t(h_1) \circ \ldots\circ  t(h_{k-1})(y_k) | \dots  \;\dots     
              |  t(h_1)  \circ \ldots \circ t(h_{n-1})(y_n)] \in \locm
\end{align*}
which proves the existence of $\5{\7w}_k$.
\medskip

Steps 1 and 2 prove part (c) in the statement. On the one hand, for any $x \in N(\locm)$ and any $g \in N(\loch)$,  $(x,g)=(x,1)(1,g) \in N(\loce)$. On the other hand, if $(x,g)$ is an arbitrary element of $N(\loce)$,  then we also have $(x,g)(1,g^{-1})\in N(\loce)$, but $(x,g)(1,g^{-1}) = (x \eta(g,g^{-1}), 1) $. In particular $x \eta(g,g^{-1})\in N(\locm)$. Since, $ \eta(g,g^{-1})\in N(\locm)$,  it follows that $x\in N(\locm)$.
\end{proof}

\begin{rmk}\label{rmkconj}
Let $\loce = \locm \times_{\phi} \loch$ be a twisted cartesian product, and let $(x, g), (y, h) \in \loce$ be such that $\7u = [(x,g)|(y,h)|(x,g)^{-1}] \in \loce$. Recall the formulae \eqref{cond3} and \eqref{cond4} for the product and inversion in $\locm$ respectively. Then,
\begin{align*}
\Pi(&\7u) = (x, g) \cdot (y, h) \cdot \big(\eta(g^{-1},g)^{-1} \cdot t(g^{-1})(x^{-1}), g^{-1}\big) = \\
 & = \big(x \cdot t(g)(y) \cdot \eta(g,h), g \cdot h\big) \cdot \big(\eta(g^{-1},g)^{-1} \cdot t(g^{-1})(x^{-1}), g^{-1}\big) = \\
 & = \big(x \cdot t(g)(y) \cdot \eta(g, h) \cdot t(gh)(\eta(g^{-1},g)^{-1}) \cdot (t(gh) \circ t(g^{-1}))(x^{-1}) \cdot \eta(gh, g^{-1}), g \cdot h \cdot g^{-1}\big) = \\
 & = \big(x \cdot t(g)(y) \cdot \eta(g, h) \cdot \eta(ghg^{-1}, g) \cdot (\9{\eta(gh, g^{-1})^{-1}}(t(gh) \circ t(g^{-1}))(x^{-1})), g \cdot h \cdot g^{-1}\big) = \\
 & = \big(x \cdot t(g)(y) \cdot \eta(g, h) \cdot \eta(ghg^{-1}, g)^{-1} \cdot t(ghg^{-1})(x^{-1}), g \cdot h \cdot g^{-1}\big),
\end{align*}
where the equality between lines three and four follows from property (b) of twisting pairs (see Definition \ref{rmktwistp}), and the equality between lines four and five follows since $\eta(ghg^{-1},g)$ defines a homotopy $\big(t(ghg^{-1}) \Left8{\eta(ghg^{-1},g)} t(gh) \circ t(g^{-1})\big)$ (see Lemma \ref{tpair}). In particular, if $(y,h) \in \locm$, then $h = 1$ and the above computation simplifies to
$$
(x,g)\cdot (y,1) \cdot (x,g)^{-1} = (x \cdot t(g)(y) \cdot x^{-1}, 1) 
$$
by property (a) of twisting pairs (see Definition \ref{rmktwistp}). Notice that this also proves that $\locm$ is a partial normal subgroup of $\loce$, in the sense of \cite[Definition 3.4]{Chermak}.

\end{rmk}

\begin{rmk}\label{rmkconj2}
According to Theorem~\ref{twistedn-simpl}(b) a extension of partial groups 
$\locm \Right1{\iota} \loce \Right1{\tau} \loch$, is determined up to equivalence by a twisting pair $\phi=(t,\eta)$. 
Now, according to Lemma~\ref{NEL} and Remark~\ref{rmkconj} for each $g\in \loch_1$, $t(g)\in \Aut(\locm)$ can be 
realized as conjugation in $\loce$ by a preimage of $g$ in $N_\loce(\locm)$. Furthermore, for 
a different choice of twisting pair 
$(t',\eta')$, $t'$ would differ from $t$ by a conjugation automorphism $\7c_z\in \Aut(\locm)$ with 
$z\in N(\locm)$. In particular, the outer action 
$$\alpha\colon \loch\Right0{}B\Out(\locm)\,,$$ 
defined 
$\alpha(g)=[t(g)] \in \Out(\locm)$, for $g\in \loch_1$,  is uniquely determined by the equivalence 
class of the extension. We
will refer to it as the outer action induced by the extension. 
\candrop{By \cite[Theorem IV.5.6]{BGM}, homotopy classes of maps $B \Right0{} B\aut(F)$ classify fibre bundles $f \colon E \Right0{} B$ with fibre $F$ up to strong homotopy equivalence. Thus, every fibre bundle $f \colon E \Right0{} B$ with fibre $F$ naturally induces a map $B \Right0{} B\pi_0(\aut(F))$. If both $B = \loch$ and $F = \locm$ are partial groups, then by Corollary \ref{isoaut2} we deduce that every fibre bundle with base $\loch$ and fibre $\locm$ induces a homomorphism of partial groups
$$
\alpha \colon \loch \Right1{} B\Out(\locm),
$$
or, in terms of Definition \ref{defiaction}, an outer action of $\loch$ on $\locm$.} We will denote $\cale(\locm, \loch, \alpha)$ the set of equivalence classes of extensions of 
$\loch$ by $\locm$ with induced outer action $\alpha\colon\loch\Right0{}\Out(\locm)$.

\end{rmk}


After the structure theorem for extensions of partial groups has been established, the classification 
of extensions goes exactly as in the case of finite groups. We will state the result and give the necessary 
hints to follow the classical argument.

\begin{thm}\label{classext}

Let $\locm$ and $\loch$ be partial groups, and let $\alpha \colon \loch \to B\Out(\locm)$ 
be an outer action. Then, the following holds.
\begin{enumerate}[\rm (a)]

\item The set $\cale(\locm, \loch, \alpha)$ is nonempty if and only if a 
certain obstruction class $[\kappa] \in H^3(\loch; Z(\locm))$ vanishes.  

\item If the set $\cale(\locm, \loch, \alpha)$ of equivalence classes of extensions of 
$\loch$ by $\locm$ with outer action $\alpha$ is not empty, then  $H^2(\loch; Z(\locm))$ 
acts freely and transitively on this set. 

\end{enumerate}
\end{thm}

\begin{proof} This result follows easily from the classificacion theorem for fibre 
bundles \cite[5.6]{BGM}. The set $\cale(\locm, \loch, \alpha)$ 
of equivalence classes of extensions of $\loch$ by $\locm$ is in bijection with 
the set $[\loch, B\aut(\locm)]$ of homotopy classes of maps from $\loch$ to 
the classifying space $B\aut(\locm)$. 
Among those, the ones with outer action $\alpha$, correspond to the classes of 
maps $\loch\Right0{} B\aut(\locm)$ lifting the map $\alpha\colon \loch\Right0{}B\Out(\locm)$. 
The result follows from the obstructions to existence and classification of such liftings. 

%
%
%
%
%
%
%
%
%

However, the description of extensions provided by Theorem~\ref{twistedn-simpl}, allows an alternative 
proof of the theorem along the lines of classification theorem for extension of groups (cf.~\cite{McL,Br}). 
We give here the general outline, but omit the lengthy calculations. 

\begin{enumerate}[(1)]

\item Fix the outer action $\alpha \colon \loch \to B\Out(\locm)$ and choose a representative $t(g) \in \Aut(\locm)$ for each $[g] \in \loch_1$, with $t(1) = \Id$. This defines a function $t \colon \loch_1 \to \Aut(\locm)$.

\item Given a 2-simplex $[g|h] \in \loch_2$, the composition $t(g) \circ t(h) $ differs 
from $t(gh)$ by an inner automorphism.
Hence, there is a choice of elements $\eta(g,h)\in N(\locm)$  satisfying
$$
t(g)\Left4{\eta(g,h)}   t(gh)t(h)^{-1}\,.
$$
Since $t(1) = \Id$, one may choose $\eta(g,1) = 1 = \eta(1,g)$ for all $[g] \in \loch_1$. 
This defines a function $\eta \colon \loch_2 \to N(\locm)$. 
Note that the pair $\phi = (t, \eta)$ already satisfies conditions (a) and (b) in Definition \ref{rmktwistp}. 
The failure of condition (c) is measured by 
$$ 
    \kappa([g|h|k])\defeq  \eta(g,h) \cdot \eta(gh,k)\cdot \eta(g, hk)^{-1}\cdot  t(g)(\eta(h,k))^{-1}\,.
$$
Notice that $\kappa([g|h|k])$ defines a self-homotopy of the identity, hence $\kappa([g|h|k])\in Z(\locm)$, so 
$\kappa\colon \loch_3\Right0{} Z(\locm)$ represents a class $[\kappa]\in H^3(\loch;Z(\locm))$.

Now following the classical argument, if $\kappa$ is a boundary, we can modify $\eta$ so that it also 
satisfies condition (c) in Definition \ref{rmktwistp}, hence $(t,\eta)$ is a well defined twisting pair 
providing an extension of $\loch$ by $\locm$. Conversely, if such an extension exists, 
then in particular there is a twisting pair by Theorem~\ref{twistedn-simpl}(b). 
In particular, 
as condition (c) in Definition~\ref{rmktwistp}  is already satisfied, this means that 
$\kappa$ already represents the trivial class in $H^3(\loch; Z(\locm))$.

\item 
Assume that the set  of equivalence classes of extensions of $\loch$ by $\locm$
 with outer action $\alpha$ is not empty. Given an extension defined by a twisting pair
 $(t, \eta)$ and a class $[v]\in H^2(\loch;Z(\locm))$, represented by 
 $v\colon \loch_2\Right0{}Z(\locm)$, with $v(g,1)=v(1,g)=1$, 
 we can form a new twisting pair $(t, \sigma)$,  with 
 $$ \sigma([g|h]) = v(g,h)\cdot \eta(g,h)$$
and $(t,\sigma)$ is a new well defined twisting pair, providing a new extension. 
This construction defines an action of $ H^2(\loch;Z(\locm))$ on the set of 
equivalence classes of extensions of $\loch$ by $\locm$
 with outer action $\alpha$. It turns out that this action is free and transitive. The proof follows 
 the classical argument.\hfill\qed
\end{enumerate}
\renewcommand{\qed}{}
\end{proof}


\section{Regular split extensions}\label{regsplit}

In this section we specialize our study to split extensions. 
When studying group extensions, a splitting always gives rise to a semidirect product. 
The subtleties involved in the study of partial groups - and their extensions - imply that 
this is no longer true in this case. Thus, in this section we introduce a particular type 
of sections, \emph{regular sections}, and we present a classification of 
\emph{regular split extensions} of partial groups, which naturally generalizes 
the classification of split group extensions.

\begin{defi}\label{defireg}

Let $\locm \Right0{} \loce \Right1{\tau} \loch$ be an extension of partial groups, and let $\locn_{\loce}(\locm) \leq \loce$ be the normalizer partial subgroup of $\locm$ in $\loce$. A section $\sigma\colon \loch \Right0{} \loce$ is \emph{regular} if it factors through the  $\locn_{\loce}(\locm)$; that is, if $\sigma(\loch) \subseteq \locn_{\loce}(\locm)$. The extension $\tau$ is \emph{regular split} if it admits a regular section.

\end{defi}

\begin{expl}\label{semidir}

The canonical example of a regular split extension is given by the 
\emph{semidirect product} of partial groups. Let $\locm$ and $\loch$ be partial groups, and let $\rho \colon \loch \to B\Aut(\locm)$ be a homomorphism of partial groups. Then, $\rho$ determines a tautological twisting pair $\phi_{\rho} = (\rho, \underline{1})$, where $\underline{1}$ stands here for the trivial cocycle $\underline{1}(g, h) = 1 \in \locm$, for all $[g|h] \in \loch_2$. The semidirect product is then defined as
$$
\locm \rtimes_{\rho} \loch \defeq \locm \times_{\phi_{\rho}} \loch.
$$
According to Theorem~\ref{twistedn-simpl}, $\locm \rtimes_{\rho} \loch$, can be described as the partial group with $n$-simplices $[(x_1,g_1)|\ldots|(x_n,g_n)]$ such that $[g_1|\ldots|g_n] \in \loch_n$ and such that $[x_1|\9{g_1}x_2|\9{g_1g_2}x_3|\ldots|\9{g_1\cdots g_{n-1}}x_n] \in \locm_n$, where $\9{g}x = \rho(g)(x)$ for each $x \in \locm_1$ and each $g \in G$, and multiplication 
determined by the formula 
$$(x_1,g_1)\cdot (x_2,g_2) = (x_1 \cdot \9{g_1}x_2, g_1 \cdot g_2)\,.
$$
Note that the notation $\9g x = \rho(g)(x)$ is coherent with the notation introduced in  
\ref{not-conj} for conjugations in a partial group.

The vertex of $\locm$ is fixed by the action of $\loch$. This provides a canonical  section 
\begin{equation}\label{cansec}
\sigma_0\colon [g_1|\ldots|g_n] \in \loch_n \longmapsto  [(1,g_1)|\ldots|(1,g_n)] \in  \locm \rtimes_{\rho} \loch
\end{equation}
Furthermore, the section $\sigma_0$ described above is regular in the sense of Definition \ref{defireg}. When there is no place for confusion we   simply use the notation $\locm \rtimes \loch$, without further reference to the action $\rho$ of $\loch$ on $\locm$.

\end{expl}

The next result shows that all  regular split fibre bundles are actually isomorphic to semidirect products.

\begin{lmm}\label{rsplitaction}

Let $\locm \to \loce \to \loch$ be a regular split extension of partial groups. Then, a regular section $\sigma$ determines a homomorphism of partial groups $\rho \colon \loch \to B\Aut(\locm)$ and an isomorphism of regular split extensions
\begin{equation}\label{splitext}
\vcenter{
 \xymatrix{  \locm \ar@{=}[r]  \ar[d]& \locm \ar[d]\\
      \locm \rtimes_\rho \loch \ar[r]^-{\varphi}  \ar[d]  & \loce  \ar[d]\\
         \loch \ar@/^2ex/[u]^{\sigma_0} \ar@{=}[r]  & \loch  \ar@/^2ex/[u]^{\sigma}
}
}
\end{equation}
where $\sigma_0\colon \loch \Right0{} \locm \rtimes_\rho \loch $ is the canonical section of the semidirect product.

\end{lmm}

\begin{proof}

According to Theorem~\ref{twistedn-simpl}, we can assume that $\loce = \locm\times_\phi \loch$ for a $\loch$-twisting pair $\phi=(t,\eta)$ for $\locm$. Consider the induced extension $BN(\locm) \to \locn_{\loce}(\locm) \to \loch$ from Lemma \ref{NEL} (a). Since the extension is regular split, there is a section $\sigma \colon \loch \to \locn_{\loce}(\locm)$. Thus, for each $g \in \loch_1$ we can express
$$
\sigma(g) = (\theta(g), g),
$$
where $\theta(g) \in N(\locm)$. Furthermore, since $\sigma \colon \loch \to \locn_{\loce}(\locm)$ is a homomorphism of partial groups, for all $[g|h] \in \loch_2$ we have
$$
(\theta(gh), gh) = \sigma(gh) = \sigma(g)\sigma(h)  = (\theta(g) \cdot t(g)(\theta(h)) \cdot \eta(g,h), gh)\,.
$$
As a consequence, we deduce that
\begin{equation}\label{trhob}
\eta(g,h) = t(g)(\theta(h)^{-1}) \cdot \theta(g)^{-1} \cdot \theta(gh)\,.
\end{equation}

Let $g \in \loch_1$, and let $\sigma(g) = (\theta(g), g)$ as above. As $\theta(g) \in N(\locm)$, we know that conjugation by $\theta(g)$ defines an automorphism of $\locm$ by Proposition \ref{Nloclgroup}, and we may define
\begin{equation}\label{trhoa}
\rho(g) = \7c_{\theta(g)} \circ t(g) \in \Aut(\locm)\,.
\end{equation}
We have to check that this defines a homomorphism of partial groups. Assuming that, notice that $\rho$ clearly lifts the outer action of $\loch$ on $\locm$ given by the $\loch$-twisting pair $\phi = (t, \eta)$. Given $[g|h] \in \loch_2$, it is enough to show that $\rho(g) \circ \rho(h) = \rho(gh)$:
\begin{align*}
\rho(g) \circ \rho(h) & = \cj{\theta(g)} \circ t(g) \circ \cj{\theta(h)} \circ t(h) = \\
 & = \cj{\theta(g)} \circ \cj{t(g)(\theta(h))} \circ t(g) \circ t(h) = \\
 & = \cj{\theta(g)} \circ \cj{t(g)(\theta(h))} \circ \cj{\eta(g,h)} \circ t(gh) = \\
 & = \cj{\theta(gh)} \circ t(gh) = \rho(gh).
\end{align*}
Here, the equality between lines two and three follows by condition (a) in Definition \ref{rmktwistp} for twisting pairs, and the equality between lines three and four follows from equation \eqref{trhob}, as $\theta(gh) = \theta(g) \cdot t(g)(\theta(h)) \cdot \eta(g, h)$.

Equations \eqref{trhob} and \eqref{trhoa} in particular imply that the section $\sigma$ 
determines the $\loch$-twisting pair $\phi$, and this makes the statement of the Lemma feasible. 
Notice, that there is a unique possible definition of a map 
$\varphi \colon \locm \rtimes_{\rho} \loch \to \loce$ such that the diagram \eqref{splitext} commutes: 
$$
\varphi(x,g) = (x \cdot \theta(g), g).
$$
We   first show that this definition extends to a well defined simplicial map. Unlike the classical case of group extensions, the main problem now is to show that this is well defined, namely that given an arbitrary simplex $\7w = [(x_1,g_1)| (x_2, g_2)|\ldots | (x_n, g_n)] 
\in \locm\rtimes_\rho \loch$, $x_i\in \locm$, $g_i\in \loch_1$, the following is a well defined simplex in $\loce$:
$$
\7w' = [(x_1\theta(g_1), g_1) | (x_2 \theta(g_2), g_2)| \ldots | (x_n  \theta(g_n), g_n)] \in \loce.
$$

We have to check conditions (i) and (ii) in Theorem \ref{twistedn-simpl} (a) for $\7w'$. Notice that condition (i) follows immediately, as it is already satisfied for $\7w$, and we have to show that
\begin{equation}\label{bla0}
[x_1 \theta(g_1)|t(g_1)(x_2 \theta(g_2))|(t(g_1) \circ t(g_2))(x_3 \theta(g_3))|\ldots |(t(g_1) \circ \ldots \circ t(g_{n-1}))(x_n \theta(g_n))] \in \locm.
\end{equation}
For $1 \leq i \leq j \leq n$, set $g_{i,j} = g_i \cdot \ldots \cdot g_j$. Condition (ii) in Theorem \ref{twistedn-simpl} (a) for $\7w$ yields
$$
[x_1|\rho(g_1)(x_2)|\rho(g_{1,2})(x_3)|\ldots|\rho(g_{1,n-1})(x_n)] \in \locm\,,
$$
and since $\rho(g) = \7c_{\theta(g)}\circ t(g)$, we have 
\begin{equation}\label{bla1}
\begin{split}
[x_1|\theta(g_1) \cdot t(g_1)(x_2) \cdot \theta(g_1)^{-1}| & \theta(g_{1,2}) \cdot t(g_{1,2})(x_3) \cdot \theta(g_{1,2})^{-1}|\ldots \\
& \ldots|\theta(g_{1,n-1}) \cdot t(g_{1,n-1})(x_n) \cdot 
\theta(g_{1,n-1})^{-1}] \in \locm.
\end{split}
\end{equation}
For each $[g|h] \in \loch_2$ the element $\eta(g,h)$ determines a homotopy $t(g) \circ t(h) \Left2{\eta(g,h)} t(gh)$ (this is an immediate consequence of condition (b) in Definition \ref{rmktwistp} and Lemma \ref{ho.eq} (b)). Thus, for each $1 \leq i \leq n$ we have
\begin{equation}\label{bla2}
t(g_{1,i}) = \cj{\eta(g_{1,i-1},g_i)^{-1}} \circ \cj{\eta(g_{1, i-2}, g_{i-1})^{-1}} \circ \ldots \circ \cj{\eta(g_1, g_2)^{-1}} \circ t(g_1) \circ \ldots \circ t(g_i).
\end{equation}
Successive applications of the above formula to \eqref{bla1} imply that $\7u = [y_1|\ldots |y_n] \in \locm$, where $y_1 = x_1$, $y_2 = \theta(g_1) \cdot t(g_1)(x_2) \cdot \theta(g_1)^{-1}$, and
\begin{align*}
y_j = \theta(g_{1, j-1}) \cdot & \eta(g_{1, j-2}, g_{j-1})^{-1} \cdot  \ldots \cdot \eta(g_1, g_2)^{-1} \cdot \\
 & \cdot (t(g_1) \circ \ldots \circ t(g_{j-1}))(x_j) \cdot \eta(g_1, g_2) \cdot \ldots \cdot \eta(g_{1,j-2}, g_{j-1}) \cdot \theta(g_{1, j-1})^{-1}
\end{align*}
for $j = 3, \ldots, n$.

Recall that all the elements of the form $\theta(g)$ or $\eta(g, h)$ belong to $N(\locm)$. Applying Lemma \ref{N(M)conj} (a) to $\7u$ above (and the appropriate products), we obtain a simplex $\7v = [z_1|\ldots|z_n] \in \locm$, where $z_1 = x_1 \cdot \theta(g_1)$, $z_2 = t(g_1)(x_2) \cdot \theta(g_1)^{-1} \cdot \theta(g_{1,2}) \cdot \eta(g_1, g_2)^{-1}$, and
\begin{align*}
z_j = (t(g_1) \circ & \ldots \circ t(g_{j-1}))(x_j) \cdot \eta(g_1, g_2) \cdot \eta(g_{1,2}, g_3) \cdot \ldots \cdot \eta(g_{1, j-2}, g_{j-1}) \cdot \\
 & \cdot \theta(g_{1, j-1})^{-1} \cdot \theta(g_{1, j}) \cdot \eta(g_{1,j-1}, g_j)^{-1} \cdot \eta(g_{1, j-2}, g_{j-1})^{-1} \cdot \ldots \cdot \eta(g_1, g_2)^{-1}
\end{align*}
for $j = 3, \ldots, n$. In the particular case of $j = n$, notice that we are adding some terms on the right of $z_n$, which were not present in $y_n$. These extra terms are all elements of $N(\locm)$ and thus this is justified. To finish the proof of \eqref{bla0}, note that
$$
z_j = (t(g_1) \circ \ldots t(g_{j-1}))(x_j \theta(g_j))
$$
for each $j = 1, \ldots, n$ (with $z_1 = x_1 \theta(g_1)$). Indeed, to show the above it is enough to apply (as many times as necessary) the formulas \eqref{trhob} and \eqref{bla2} to  $z_j$.

This shows that $[(x_1\theta(g_1), g_1) | (x_2 \theta(g_2), g_2)| \ldots | (x_n  \theta(g_n), g_n)] \in \loce$. To finish the proof, we have to check that $\varphi$ defined above is a simplicial map, and an isomorphism of partial groups. To show that it is a simplicial map, it is enough to show  that $\varphi(\Pi[(x,g)|(y,h)]) = \Pi(\varphi[(x,g)|(y,h)])$. We have
 $$\varphi(\Pi[(x,g)|(y,h)])  = \varphi([x \cdot  \rho(g)(y), gh]) = [x\cdot  \rho(g)(y)\cdot \theta(gh), g\cdot h]$$
 and 
\begin{align*}
 \Pi(\varphi[(x,g)|(y,h)]) &= \Pi[(x\cdot \theta(g),g)| (y\cdot \theta(h),h)] \\
                                       &   = [(x\cdot \theta(g)\cdot  t(g)(y)\cdot t(g)(\theta(h))\cdot \eta(g,h), g\cdot h)]  \\
     & = [(x\cdot \rho(g)(y)\cdot \theta(g)\cdot t(g)(\theta(h))\cdot \eta(g,h), g\cdot h)] \\
& = [(x\cdot \rho(g)(y)\cdot \theta(gh), g \cdot h)]  
\end{align*}
where the last two equalities follow by (\ref{trhoa}) and (\ref{trhob}). Finally, notice that by definition of $\varphi$ it restricts to the identity on $\locm$ and also induces the identity on $\loch$. It follows that $\varphi$ itself is a bijection. 
\end{proof}

Notice that, conversely, given a homomorphism $\loch \to B\Aut(\locm)$, and a fibrewise isomorphism $\varphi\colon \locm\rtimes_\rho \loch \Right0{} \loce$, we can define the section $\sigma = \varphi\circ \sigma_0$, with image in $\locn_{\loce}(\locm)$.


\subsection{Classification of sections}

We present a classification of sections of regular split extensions, based on an adapted version of non-abelian cohomology of groups (cf.~\cite{Ser}). 
%
%
%
%
%
%
%

Fix a regular split extension
$\locm \Right2{\iota} \loce \Right2{\tau} \loch$. We   denote by 
$\Gamma(\tau)$  the set of all  sections of $\tau$. 
Two sections $\sigma, \sigma' \in \Gamma(\tau)$ are \emph{vertically homotopic} 
if there is a homotopy $\sigma \Left1{j(y)}\sigma'$, with $y\in \locm_1$.  
As in the classical case of Hurewicz fibrations, two sections that are homotopic are indeed 
vertically homotopic. This is proved in Lemma~\ref{equivsec-b}, below. 
But there is an additional problem. 
Partial groups are simplicial sets that fail to be Kan complexes, unless they are actual groups. 
In this situation homotopy does not generally define an equivalence relation. 
In particular, homotopy or vertical homotopy as described above,  is reflexive and symmetric, 
but it is not generally transitive. We are therefore forced to take the equivalence relation
that it generates. 


\begin{defi}\label{equivsec1} 
Let $\locm \Right0{\iota} \loce \Right0{\tau} \loch$ be an extension of partial groups. We say that two sections $\sigma,\sigma'\in\Gamma(\tau)$ are equivalent,
$\sigma\approx\sigma'$, if there is a string of sections and homotopies
 $$\sigma_0 \Left1{x_1} \sigma_1 \Left1{x_2} \ldots \Left1{x_n} \sigma_n$$
 with $x_i\in \loce_1$, $ i=0,\dots,n$.  We define $\frakx(\tau)$ as the set of equivalence classes $\Gamma(\tau)/{\approx}$.
\end{defi}

\begin{lmm}\label{EconjLconj}

Let $\locm \Right1{\iota} \loce \Right1{\tau} \loch$ be an extension, and let $\sigma_0$ be a section. 
Assume in addition that there exist homomorphisms of partial groups 
$\sigma_i \colon \loch \to \loce$ and elements $x_i \in \loce_1$, for $i = 1, 2$, such that 
$\sigma_0 \Left1{x_1} \sigma_1 \Left1{x_2} \sigma_2$. Given an element $g \in N(\loch)$, 
set $\sigma'_i = \sigma_i \circ \7c_g$, for $i = 1, 2$. Then, there are homotopies
$$
\sigma_0 \Left2{x_1'} \sigma_1' \Left2{x_2'} \sigma_2,
$$
where $x_1' = x_1 \cdot \sigma_1(g)$ and $x_2' = \sigma_1(g^{-1}) \cdot x_2$. 
In particular, if $g = \tau(x_1)^{-1}$, then $\sigma_1'$ is a section of $\tau$.
\end{lmm}

\begin{proof}
By definition of $N(\loch)$, an element $g\in \loch$ determines a homotopy 
$\Id_\loch\Left1g\7c_g$, hence according to Lemma~\ref{ho.eq}(d), also 
$\sigma_0 \Left4{x_1\cdot \sigma_1(g)} \sigma_1\circ \7c_g$
and similarly, $\sigma_1\circ \7c_g  \Left5{\sigma_1(g^{-1})\cdot x_2} \sigma_2$.

By Lemma~\ref{ho.eq}(c), we obtain a homotopy 
$\Id_{\loch} = \tau\circ\sigma_0 \Left3{\tau(x_1)} \tau\circ \sigma_1$, so 
by definition of $N(\loch)$, $\tau(x_1)\in N(\loch)$ and $\7c_{\tau(x_1)} = \tau\circ \sigma_1$. 
Hence, if $g=\tau(x_1)^{-1}$, then $\tau\circ \sigma_1' = \tau\circ \sigma_1\circ \7c_g=
\7c_{\tau(x_1)}\circ \7c_g = \Id_\loch$, by Lemma~\ref{N(M)conj}(\ref{Normaux(c)}).
\end{proof}

\begin{lmm}\label{equivsec-b}

Let $\locm \Right1{\iota} \loce \Right1{\tau} \loch$ be an extension of partial groups, and let $\sigma, \sigma' \in \Gamma(\tau)$. Then, the following conditions are equivalent.
\begin{enumerate}[\rm (a)]

\item There are homomorphisms of partial groups $\sigma_0 = \sigma, \sigma_1, \ldots, \sigma_n = \sigma' \colon \loch \to \loce$ and elements $x_1, \ldots, x_n \in \loce_1$ such that $\sigma_0 \Left1{x_1} \sigma_1 \Left1{x_2} \ldots \Left1{x_n} \sigma_n$.

\item There are sections $\sigma_0 = \sigma, \sigma_1, \ldots, \sigma_n = \sigma' \in \Gamma(\tau)$ and elements $x_1, \ldots, x_n \in \loce_1$ such that 
$\sigma_0 \Left1{x_1} \sigma_1 \Left1{x_2} \ldots \Left1{x_n} \sigma_n$.

\item There are sections $\sigma_0 = \sigma, \sigma_1, \ldots, \sigma_n = \sigma' \in \Gamma(\tau)$ and elements $y_1, \ldots, y_n \in \locm_1$ such that 
$\sigma_0 \Left1{\iota(y_1)} \sigma_1 \Left1{\iota(y_2)} \ldots \Left1{\iota(y_n)} \sigma_n$.

\end{enumerate}

\end{lmm}

\begin{proof}

Assume that (a) holds.  Note that $\sigma_0 \in \Gamma(\tau)$ by assumption. Suppose that $\sigma_j \in \Gamma(\tau)$ for $j = 0, \ldots, i-1$, for some $i \geq 1$. 
Composing with $\tau$, we get  $\Id_{\loch}=\tau\circ\sigma_{i-1} \Left2{\tau(x_i)} \tau\circ \sigma_i$, 
hence, if we denote  $g=\tau(x_i)$, we have  $\tau \circ \sigma_i = \7c_{g^{-1}}$. 
It follows that $\sigma_i' = \sigma_i \circ \7c_g$ is a regular section of $\tau$, and by Lemma \ref{EconjLconj} there is a sequence
$$
\sigma_0 \Left2{x_1} \ldots \Left2{x_{i-1}} \sigma_{i-1} \Left2{x_i'} \sigma_i' \Left2{x_{i+1}'} \sigma_{i+1} \Left2{x_{i+2}} \ldots \Left2{x_n} \sigma_n,
$$
where $x_i' = x_i \cdot \sigma_i(g)^{-1}$ and $x_{i+1}' = \sigma_1(g) \cdot x_{i+1}$. Then, we can apply the same argument repeatedly and show that (b) holds.
\smallskip

Suppose next that condition (b) holds. By assumption, we have a string of homotopies
$$
\sigma_0 \Left1{x_1} \sigma_1 \Left1{x_2} \ldots \Left1{x_n} \sigma_n,
$$
where $\sigma_0, \ldots, \sigma_n \in \Gamma(\tau)$, and $x_1, \ldots, x_n \in \loce_1$. 
We show that each of these homotopies can be modified to an homotopy 
$\sigma_{i-1} \Left1{j(y_i)} \sigma_i$, for some $y_i\in \locm_1$. We proceed by induction. 
suppose that $x_1= j(y_i)\,,\,\dots\,,\,x_{i-1}= j(y_{i-1})$.
  
Composing the $i$th homotopy with $\tau$, we obtain a homotopy 
$\Id_{\loch}\Left2{\tau(x_i)}\Id_{\loch}$. Thus, $g= \tau(x_i)^{-1}$ determines a self-homotopy of the 
identity on $\loch$, hence $g\in Z(\loch)$. 
Notice $x_i'=x_i \cdot \sigma_i(g)$ belongs to the fibre over the vertex of $\loch$, that is,
there is $y_i\in\locm_1$ with $\iota(y_i)=x_i \cdot\sigma_i(g)$. Furthermore, $ \sigma_i \circ \7c_g = \sigma_i$, hence, 
applying Lemma \ref{EconjLconj} to $\sigma_{i-1} \Left1{x_i} \sigma_i \Left1{x_{i+1}} \sigma_{i+1}$ 
 we obtain a new string 
$$
\sigma_0 \Left4{\iota(y_1)}\dots \ \dots\sigma_{i-1}\Left4{\iota(y_i)} 
\sigma_i \Left4{ \sigma_i(g^{-1})x_{i+1}  } 
\sigma_{i+1} \Left4{x_{i+2}}\ldots \  \ldots \Left4{x_n} \sigma_n,
$$
By an iteration of this argument, we prove that (c) holds. Finally, it is clear that if condition (a) 
holds, then so does condition (c).
\end{proof}

\subsection{Derivations}

Let $\loch$ and $\locm$ be partial groups. 
Given an action  $\rho\colon \loch\Right0{}B\Aut(\locm)$ of $\loch$ on $\locm$, 
we will write $\9h\7x= \rho(h)(\7x)$, for simplicity, without further mention of $\rho$, when no confusion is possible (see Example \ref{semidir} for a 
justification of this notation).

\begin{defi}\label{defideriv}
A \emph{derivation} $\theta\colon \loch \Right0{}\locm$ is a 
collection of maps $\{ \theta_n\colon \loch_n\Right0{}\locm_n\}_{n\geq0}$ satisfying
\begin{enumerate}[\rm(i)]
\item $d_0(\theta_n(\7h))= \9{h_1}\theta_{n-1}(d_0\7h)$ and 
 $d_i(\theta_n(\7h))= \theta_{n-1}(d_i\7h)$, $1\leq i\leq n$, 
\item $s_i(\theta_n(\7h))= \theta_{n+1}(s_i\7h)$, $0\leq i\leq n$, 
\end{enumerate}
for all $\7h=[h_1|\dots|h_n]\in \loch_n$, $n\geq0$. 
 We denote by $\Der(\loch,\locm)$  the set of derivations from $\loch$ to $\locm$.

\end{defi}
$\Der(\loch,\locm)$ is a pointed set, with canonical base point defined as the constant derivation $c$, 
$c(h)=1$, for all $h\in\loch_1$.
It turns out that a derivation is uniquely determined by its effect in dimension 1, 
$\theta=\theta_1\colon \loch_1 \Right0{}\locm_1$ by the formula 
$$
\theta_n([h_1|h_2|\dots| h_n]) =
    \bigl[  \theta(h_1) | \9{h_1}\theta(h_2) | \9{h_1h_2}\theta(h_3)| \dots
                             | \9{h_1h_2\dots h_{n-1}}\theta(h_n)  \bigr]\,.
$$

Following the classical definitions, we 
  say that an element  $y \in \locm_1$ defines an equivalence  $\theta \sim \theta'$ between 
two derivations
$\theta, \theta'\in \Der(\loch, \locm)$,  
if for any simplex 
 $\7h=[h_1|h_2|\dots| h_n]\in \loch_n$, $n\geq 0$, there are simplices 
 $$
 \7v_k= \bigl[  \theta(h_1) | \9{h_1}\theta(h_2) |  \dots | 
                   \9{h_1h_2\dots h_{k-1}}\theta(h_k)|  \9{h_1h_2\dots h_{k}}y|
                        \9{h_1h_2\dots h_{k}}\theta'(h_{k+1})|  \dots
                             | \9{h_1h_2\dots h_{n-1}}\theta'(h_n)  \bigr]\in\locm_{n+1}\,,
$$
$ k=0,\dots ,n$, with $\Pi(\7v_0) =\Pi(\7v_1) = \dots = \Pi(\7v_n)$.

We extend $\sim$ to an equivalence relation and define
$$H^{1}(\loch,\locm)= \Der(\loch,\locm)/{\sim}\,. $$ 


\medskip

\begin{defi}
 A \emph{regular derivation} is a derivation $\theta\colon \loch\Right0{}\locm$ with image in $BN(\locm)$. 
\end{defi}
The characterization of regular derivations is simplified by the fact that 
$N(\locm)$ is a group. Indeed, a 
regular derivation reduces to a map $\theta\colon \loch_1\Right0{} N(\locm)$ satisfying
$$
\theta(g\cdot h) = \theta(g) \cdot \9{g}\theta(h)
$$
for all $[g|h]\in \loch_2$. 

\begin{thm}\label{sections-derivations}
Let $\locm\Right1{\imath}\loce\Right1{\tau}\loch$ be an extension of  partial groups.
 \begin{enumerate}[\rm (a)]
\item $\tau$ admits a regular section if and only if a given obstruction class 
 $[\eta]\in H^2(\loch;N(\locm))$ vanishes.
\item A regular section defines an action $\rho\colon \loch\Right0{}\Aut(\locm)$, given by conjugation in $\loce$, and a bijective correspondence between 
the sets $H^1(\loch; \locm)$ and $\frakx(\tau)$. Furthermore, under this bijection regular derivations correspond to regular sections.
\end{enumerate}
\end{thm}

\begin{proof}
According to Theorem~\ref{twistedn-simpl}(b), there is an $\loch$-twisting function
  $\phi=(t,\eta)$ and an equivalence of extensions $\loce\cong \locm\times_\phi\loch$. 
The obstruction class to the existence of a regular section is represented by $\eta$, the second component of the $\loch$-twisting pair.  
This is already implicit in the proof of Lemma~\ref{rsplitaction}, where equation 
\eqref{trhob} presents $\eta$ as a coboundary if and only if a regular section exists. This proves (a).
Furthermore \eqref{trhoa} defines a lift of $t$ to an action 
$\rho\colon\loch\Right0{}\Aut(\locm)$. Then  Lemma~\ref{rsplitaction} provides a fibrewise 
equivalence of $\loce$ with the semidirect product 
$\locm\rtimes\loch$. Hence, to prove (b), we can compute the equivalence classes of 
sections of the projection 
$\proj\colon \locm\rtimes\loch\Right0{}\loch$.

We define maps
$$
\xymatrix{ 
\Gamma(\proj) \ar@<.5ex>[r]^-{\delta}    &  \Der(\loch, \locm) \ar@<.5ex>[l]^-{\gamma}
}
$$
as follows.
Let $\sigma \in \Gamma(\proj)$. Write $\sigma([h]) = [(\theta(h),h)]\in  \locm\rtimes\loch$, 
for each $h\in \loch_1$. Then, given a simplex $\7h[h_1|\dots|h_n]\in \loch_n$, 
$\sigma(\7h)$ is a simplex of $\locm\rtimes\loch$, written $[(\theta(h_1),h_1)|\dots |(\theta(h_n), h_n)]$. 
This implies the existence of a simplex
$[\theta(h_1)| \9{h_1}\theta(h_2)|\dots| \9{h_1h_2\dots h_{n-1}}\theta(h_n)]\in \locm_n$
(see Example~\ref{semidir}). 
So, 
$\theta$ extends to a collection of maps 
$\theta=\{\theta_n\colon \loch_n\Right0{}\locm_n\}_{n\geq0}$, with 
$\theta_n([h_1|\dots|h_n]) = [\theta(h_1)| \9{h_1}\theta(h_2)|\dots| \9{h_1h_2\dots h_{n-1}}\theta(h_n)]$. 
It is a routine to check that this is a well defined derivation, so we can define $\delta(\sigma)=\theta$.

Conversely, fix $\theta=\{\theta_n\}_{n\geq0} \in \Der(\loch, \locm)$ and define
$$
\loch \Right3{\sigma_{\theta}} \locm\rtimes \loch
$$
in degree one by $\sigma_{\theta}(g) = (\theta(g), g)$ for all $g \in \loch_1$. Again, by 
Example~\ref{semidir}, it is easy to check that this extends to a well defined homomorphism of partial groups
which is indeed a section of $\proj$. Thus we can define
$\gamma(\theta) = \sigma_{\theta}$.

These maps are inverse to each other. 
Moreover, regular sections clearly correspond to regular 
derivations. This follows at once from the identification 
$\locn_{\locm \rtimes \loch}(\locm) \cong BN(\locm) \rtimes \loch$. Indeed, recall from Lemma \ref{NEL} that there is an induced fibre bundle
$$
BN(\locm) \Right2{} \locn_{\locm \rtimes \loch}(\locm) \Right2{} \loch.
$$
Moreover, for any $h\in\loch_1$, one can easily check that 
$(1,h_1)\in N_{\locm \rtimes \loch}(\locm) $, thus $N(\locm)\rtimes\loch\leq N_{\locm \rtimes \loch}(\locm)$, hence $N(\locm)\rtimes\loch = N_{\locm \rtimes \loch}(\locm)$.


It only remains to check that $\delta$ and $\gamma=\delta^{-1}$ preserve 
the given equivalence relations 
on $\Gamma(\proj)$ and on $\Der(\loch, \locm)$. 
It is enough to show that given sections $\sigma,\sigma'$ and a homotopy 
$\sigma\Left1\eta\sigma'$, then $\theta=\delta{\sigma}$ and $\theta'=\delta(\sigma')$ are 
equivalent derivations, and vice versa. 

According to Lemma~\ref{equivsec-b}, we can assume that $\eta=(\xi,1)\in \locm\rtimes\loch$. 
Now, according to Lemma~\ref{homot=conj}, a homotopy determined by $\eta$ exists if and only if
there are  simplices
\begin{align*}
 \7w_k &= [\sigma(h_1)|\dots |\sigma(h_k)|\eta|\sigma'(h_{k+1})|\dots|\sigma'(h_n)]\\
   & =  [(\theta(h_1),h_1)|\dots |(\theta(h_k),h_k)|(\xi,1)
                     |(\theta'(h_{k+1}),h_{k+1})|\dots|(\theta'(h_n),h_n)]\in \locm\rtimes\loch
\end{align*}
for $k=0,\dots,n$, such that $\Pi(\7w_0) = \dots = \Pi(\7w_n)$.
 These, in turn,  are well defined simplices of $\locm\rtimes\loch$ if and only if 
there are simplices
\begin{align*}
 \7v_k = [\theta(h_1)| \9{h_1}\theta(h_2)|\dots|\9{h_1h_2\dots h_{k-1}}\theta(h_k)
           |\9{h_1h_2\dots h_{k}}\xi| \9{h_1h_2\dots h_{k}}\theta'(h_{k+1})
           | \dots |  \9{h_1h_2\dots h_{n-1}}\theta'(h_n)]\in \locm_{n+1}
\end{align*}
for $k=0,\dots,n$, such that $\Pi(\7v_0) = \dots = \Pi(\7v_n)$.
(Notice that $\Pi(\7w_k) = (\Pi(\7v_k),\Pi(\7h))$.)
Finally, this last condition is by definition equivalent to say that $\xi$ defines an 
equivalence $\theta\sim\theta'$. 
\end{proof}



\end{document}